\documentclass{amsart}
\usepackage{amsthm}
\usepackage{amsmath}
\usepackage{amsfonts}
\usepackage{amssymb}
\usepackage{mathrsfs}
\usepackage{mathtools}
\usepackage{tikz-cd}
\usepackage{enumitem}
\usepackage[utf8]{inputenc}
\usepackage[left=30mm,right=30mm]{geometry}
\usepackage{hyperref}
\usepackage{caption}
\usepackage{subcaption}

\theoremstyle{plain}
\newtheorem{thm}{Theorem}[section]
\newtheorem{lemma}[thm]{Lemma}
\newtheorem{prop}[thm]{Proposition}
\newtheorem{cor}[thm]{Corollary}
\newtheorem{claim}{Claim}

\newtheorem{obs}{Observation}
\newtheorem{fact}[obs]{Fact}

\newtheorem{thmintro}{Theorem}

\theoremstyle{definition}
\newtheorem{defn}[thm]{Definition}
\newtheorem{rem}[thm]{Remark}

\newtheorem{ex}[thm]{Example}

\newcommand{\mc}{\mathcal}

\newcommand{\R}{\mathbb{R}}
\newcommand{\N}{\mathbb{N}}
\newcommand{\Z}{\mathbb{Z}}

\newcommand{\acts}{\curvearrowright}

\DeclareMathOperator{\diam}{diam}

\DeclarePairedDelimiter\floor{\lfloor}{\rfloor}

\usepackage{xcolor}

\DeclarePairedDelimiter\abs{\lvert}{\rvert}%
\newcommand{\cay}{\mathrm{Cay}(G, \mc S)}

\title{Hyperbolic spaces that detect all strongly-contracting directions}

\author{Stefanie Zbinden}

\begin{document}

\begin{abstract}
    Given a geodesic metric space $X$, we construct a corresponding hyperbolic space, which we call the contraction space, that detects all strongly contracting directions in the following sense; a geodesic in $X$ is strongly contracting if and only if its parametrized image in the contraction space is a quasi-geodesic. If a finitely generated group $G$ acts geometrically on $X$, then all strongly-contracting elements act as WPD elements on the contraction space. If the space $X$ is CAT(0), or more generally Morse-dichotomous, that is if all Morse geodesics are strongly-contracting, then all generalized loxodromics act as WPD elements, implying that the action is what we call ``universally WPD''. 
\end{abstract}

\maketitle

\section*{Introduction}
Acylindrical actions were introduced by Sela \cite{Sela:acylindrical} and Bowditch \cite{Bow:curve}. Acylindrical hyperbolicity was introduced by Osin in \cite{O:acylindrical}, unifying the notions of WPD elements and hyperbolically embedded subgroups, which are of great significance. 

Acylindrical hyperbolicity has been proven to be a powerful tool \cite{BBF:cohom, BF:cohomMCG, DGO:rotating_families, H:cohom, Hull:SC_in_AH, HO:cocycles, S:random_walk, GHPRS:moarkov_chains}. Several classes of groups have been proven to be acylindrically hyperbolic, for example the mapping class groups \cite{MM:MCG}, $Out(F_n)$ \cite{BF:Out}, groups acting geometrically on a space with strongly contracting geodesics \cite{BBF:projection_complex} and many more, see for example \cite{GH:acylindrical, CW:acylindrical, LP:acylindrical, H:cohom}.

One of the most prominent open questions about acylindrical hyperbolicity is whether it is invariant under quasi-isometry or even just commensurability. This question has been asked for example in \cite[Problem 9.1]{DGO:rotating_families} and \cite[Question 2.20]{O:acylindrical}.

In \cite{MO:commensurability}, it is shown that acylindrical hyperbolicity is a commensurability invariant for groups that admit a largest acylindrical action. While not all acylindrically hyperbolic groups admit a largest acylindrical action \cite{A:non-universal-acylindrical}, this illustrates the importance of studying specific acylindrical actions. 

In \cite{PSZ:curtain}, Petyt-Spriano-Zalloum use curtains, which are generalizations of hyperplanes to, for a given CAT(0) group, construct a hyperbolic space on which it acts \emph{universally WPD} (see Definition \ref{def:universal-wpd}).

In this paper, given a geodesic metric space $X$, we associate to it a hyperbolic space $\hat{X}$ called the \emph{contraction space} (see Definition \ref{def:K-contraction-space}). We then show that groups acting geometrically on \emph{weakly Morse-dichotomous} spaces (see Definition \ref{def:morse-dich}) act universally WPD. Note that weakly Morse-dichotomous spaces are spaces where all Morse geodesics are strongly contracting, such as CAT(0)-spaces \cite{C:Morse_CAT(0)}, injective metric spaces \cite{SZ:injective}, groups hyperbolic relative to abelian subgroups \cite{S:distformrelhyp} and certain infinitely presented $C'(1/6)$--small-cancellation spaces \cite{Z:small_cancellation}.

\begin{thmintro}\label{thmintro:non-unifromv2}
    If a finitely generated group $G$ acts geometrically on a weakly Morse-dichotomous space $X$, then the action of $G$ on the contraction space $\hat{X}$ is non-uniformly acylindrical and universally WPD.
\end{thmintro}

Since CAT(0) spaces are weakly Morse-dichotomous, the theorem above can be viewed as a generalization of some of the results of \cite{PSZ:curtain}.

\subsection*{Construction of the contraction space} The moral idea behind the construction of the contraction space is as follows. If we have a geodesic metric space $X$ and we cone-off all subset which are ``not-hyperbolic'', then the resulting space $\hat{X}$ is hyperbolic. In practice, instead of coning-off non-hyperbolic subsets, we add intervals of length one between every pair of points whose connecting geodesic is not ``strongly-contracting'' enough, we call such points anti-contracting. The precise execution of these ideas and definitions can be found in Section \ref{sec:construction}. In Section \ref{sec:relatively_hyperbolic_groups} (see Proposition \ref{prop:quasi-isom-to-cone}) we show that for groups hyperbolic relative to subgroups with empty Morse boundary, the contraction space is canonically quasi-isometric to the coned-off graph, showing that the contraction space can and should be viewed as a generalization of the coned-off graph for not necessarily relatively hyperbolic groups. \\

The contraction space can be constructed for any geodesic metric space $X$, not just weakly Morse-dichotomous spaces. If there was a ``nice'' action of a group $G$ on $X$, this induces a ``nice'' action of $G$ on the contraction space $\hat{X}$. The following theorems detail consequences of different assumptions on the ``niceness'' of the action of $G$ on $X$ and the space $X$.

If we assume that a group $G$ acts properly on a geodesic metric space $X$, then the action of $G$ on the contraction space $\hat{X}$ is not necessarily a universal WPD action. However, we still know that it is a non-uniformly acylindrical action, as stated in the theorem below.

\begin{thmintro}\label{thmintro:non-uniform-acylindricity}
    Let $G$ be a group acting properly on a geodesic metric space $X$. Then the action of $G$ on the contraction space $\hat{X}$ is non-uniformly acylindrical. Further, if there exists an element $g\in G$ whose axis in $X$ is strongly contracting, then $G\acts \hat{X}$ has unbounded orbits.
\end{thmintro}

If we require that $G$ is finitely generated and acts geometrically on a geodesic metric space $X$, then we retrieve that being strongly-contracting, WPD, and loxodromic is equivalent.

\begin{thmintro}\label{thmintro:wpd-lox-contr-equivalence}
    Let $G$ be a finitely generated group which acts geometrically on a geodesic metric space $X$. Let $g\in G$ be an element. The following are equivalent; 
    \begin{enumerate}
        \item $g$ is loxodromic (w.r.t the action on $\hat{X}$),
        \item $g$ is WPD (w.r.t the action on $\hat{X}$),
        \item $g$ is strongly-contracting (w.r.t the action on $X$).
    \end{enumerate}
\end{thmintro}

In \cite{BCKM:recog_space} the notion of recognizing spaces for stable subgroups is introduced (see Definition \ref{def:recognizing-space}). Roughly speaking, a hyperbolic space $X$ is a recognizing space for a group $G$ if all stable subgroups of $G$ embed quasi-isometrically into $X$ in a natural way. The theorem below states that under nice enough assumptions, the contraction space is a universal recognizing space and that while the action of $G$ on the contraction space might not be acylindrical, $G$ acts acylindrically along its stable subgroups.

\begin{thmintro}\label{intro:stable}
    If a finitely generated group $G$ acts geometrically on a Morse-dichotomous space $X$, then the contraction space $\hat{X}$ is a universal recognizing space. 
    Moreover if $H\leq G$ is a stable subgroup, then $G$ acts acylindrically along the orbit of $H$ in $\hat{X}$.
\end{thmintro}

The theorem below is a weaker version of Theorem \ref{intro:stable} which holds in more generality. 

\begin{thmintro}\label{intro:quasi-geodesics}
    Let $X$ be a geodesic metric space. Let $C\geq 0$ be a constant and let $Y\subset X$ be a subset such that $[x, y]$ is $C$-contracting for all $x, y\in Y$. Then the inclusion $(Y, d)\hookrightarrow (\hat{X}, \hat{d})$ is a quasi-isometric embedding. Moreover, if $G$ is a group acting properly on $X$, then the action $G\acts \hat{X}$ is acylindrical along $Y$.  
\end{thmintro}

Lastly, we show that unboundedness of the contraction space comes from strongly contracting rays.

\begin{thmintro}\label{thmintro:diameter_dich}
    Let $\rho> 0$ be a constant. There exists a constant $\Delta = \Delta(\rho)$ such that the following holds. If $X$ is a geodesic metric space whose isometry group $\mathrm{Isom}(X)$ has $\rho$-dense orbit, then one of the following holds
    \begin{enumerate}
        \item \textbf{Uniform Boundedness:} $\widehat{\diam}(\hat{X})\leq \Delta$. 
        \item \textbf{Unboundedness:} $\hat{X}$ is unbounded. 
    \end{enumerate}
    Moreover, $\widehat{\diam}(\hat{X})$ is unbounded if and only if $X$ contains a strongly-contracting geodesic ray or equivalently $\mathrm{Isom}(X)\acts X$ contains a strongly-contracting element.
\end{thmintro}

We want to highlight the following structure results which are essential in the proof of the above theorems and also highlight the useful properties of the contraction space.

\begin{thmintro}\label{thmintro:structure-results}
    Let $X$ be a geodesic metric space and let $\gamma : I\to X$ be a geodesic.
    \begin{enumerate}
        \item The parametrized image $\hat{\gamma}$ of $\gamma$ in the contraction space $\hat{X}$ is a quasi-geodesic if and only if $\gamma$ is strongly contracting. \label{prop:structure-quasi-geodesics}
        \item Let $x\in X$ be a point and let $p$ and $q$ be points on $\gamma$ closest to $x$ with respect to the metrics $d$ and $\hat{d}$ respectively. We have that $\hat{d}(p, q)< 17$.\label{prop:closest-points}
        \item Let $G$ be a group acting on $X$. If $g$ acts loxodromically on $\hat{X}$, then $g$ acts strongly-contracting on $X$.
    \end{enumerate}
\end{thmintro}

\subsection*{On Petyt-Zalloum's work} In upcoming work \cite{PZ:walls}, Petyt-Zalloum generalize Sageev's construction \cite{Sageev:construction} for a non-discrete set of walls. In particular, given a geodesic metric space $X$, they construct a hyperbolic space $\tilde{X}$ such that the following hold. 
\begin{itemize}
    \item There is a coarsely Lipschitz, $\mathrm{Isom}(X)$--invariant map $p : X \to \tilde{X}$. 
    \item If $\gamma: I\to X$ is a geodesic, then $p(\gamma)$ is a quasi-geodesic if and only if $\gamma$ is strongly contracting. This is analogous to a combination of Lemma \ref{lemma:quasi-geodesic_image_implies_strongly_contracting} and Lemma \ref{lemma:contracting_implies_quasi_geodesic}.
    \item If a group $G$ acts properly on $X$ and $g\in G$, then $g\acts \tilde{X}$ loxodromic, $g\acts \tilde{X}$ WPD and $g\acts X$ strongly-contracting are equivalent. This is analogous to Theorem \ref{thmintro:wpd-lox-contr-equivalence}, and can be used to prove an analogue of the universal WPD action part of Theorem \ref{thmintro:non-unifromv2}.
    \item If $G$ acts coboundedly on $X$ and $X$ is Morse-dichotomous, then a subgroup $H\leq G$ is stable if and only its orbit in $\tilde{X}$ is quasi-isometrically embedded. This is a similar but stronger result than Theorem \ref{intro:stable}.
\end{itemize}

Their constructions also allow them to prove other results, which are unrelated to results in this paper. For example (with Spriano) they show that if $X$ is Morse-dichotomous and not hyperbolic, then every group acting coboundedly on $X$ has an infinite-dimensional space of quasi-morphisms.  

The main difference between the two constructions is that Petyt-Zalloum use and develop the machinery of spaces with walls (also leading to potential other applications, which are not about strong contraction), while in this paper, the results follow more directly from strong-contraction. 

\subsection*{Outline} In Section \ref{sec:prelim} we give background on strong-contraction and acylindrical hyperbolicity. In Section \ref{sec:construction} we define the contraction space $\hat{X}$ associated to a geodesic metric space $X$. In fact, given a geodesic metric space $X$, we construct a family of hyperbolic spaces associated to $X$. This family of hyperbolic spaces $\{\hat{X}_K\}_K$ is indexed by functions $K$ called \emph{contraction gauges}. We further prove Theorem \ref{thmintro:structure-results}\eqref{prop:structure-quasi-geodesics}. In Section \ref{sec:acylindricity} we prove Theorems \ref{thmintro:non-unifromv2} - \ref{intro:quasi-geodesics}. The key ingredient in the proofs is Theorem \ref{thmintro:structure-results}\eqref{prop:closest-points}, which allows us to control closest point projections in the contraction space. Due to Theorem \ref{thmintro:structure-results}\eqref{prop:closest-points} we can prove Lemma \ref{lemma:quadrangle-estimates}, which is the second key technical Lemma of Section \ref{sec:acylindricity}. Lemma \ref{lemma:quadrangle-estimates} states that if two opposite sides of a sufficiently large quadrangle are short in the contraction space $\hat{X}$, then the other two sides come close in $X$. In Section \ref{sec:diameter}, we prove Theorem \ref{thmintro:diameter_dich}. Lastly, in Section \ref{sec:relatively_hyperbolic_groups} we show that the contraction space of a group hyperbolic to subgroups with empty Morse boundary is naturally-isomorphic to the coned off graph. 

\subsection*{Acknowledgements} I would like to thank Harry Petyt, Davide Spriano and Abdul Zalloum for discussions about our respective projects - figuring out the similarities and differences of our constructions and their applications was very helpful. Further, I would like to thank Jacob Russell, Dawid Kielak, Matt Cordes, Ric Wade, Alexandre Martin, Antoine Goldsborough, Oli Jones and in particular my supervisor Alessandro Sisto for helpful comments and inspiring discussions both about the project and possible future applications.

\section{Preliminaries}\label{sec:prelim}

Throughout this paper, $(X, d)$ denotes a geodesic metric space. For any subspace $Y\subset X$ the closest point projection from $X$ to $Y$, if it exists, is denoted by $\pi_Y: X \to 2^Y$. Let $x\in X$. We denote the open ball of radius $r$ around $x$ by $B_r(x)$. For points $x, y\in X$ we denote by $[x, y]$ a choice of geodesic from $x$ to $y$. We assume that $[x, y] = [y, x]^{-1}$. If $\gamma$ is a geodesic and $z,z'\in \gamma$, we denote the subsegment of $\gamma$ (or its inverse) from $z$ to $z'$ by $\gamma_{[z, z']}$. In particular, if $z, z'\in [x, y]$ we denote the subsegment of $[x, y]$ from $z$ to $z'$ by $[x, y]_{[z, z']}$.
\begin{defn}
    Let $\Phi : \R_{\geq 0}\to \R_{\geq 0}$ be a function. Let $\mc A(t)$ and $\mc B(t)$ be two properties depending on a parameter $t\in \R_{\geq 0}$. We say that $\mc A$ and $\mc B$ are $\Phi$-equivalent if the following holds. If $\mc A(t)$ holds, then $\mc B(\Phi(t))$ holds and if $\mc B(t')$ holds, then $\mc A(\Phi(t'))$ holds.
\end{defn}

\begin{defn}[Quasi-geodesic]
    Let $C\geq 1$ be a constant. A continuous map $\gamma : I \to X$ is a $C$-quasi-geodesic if 
    \begin{align}
        \frac{\abs{t - s}}{C} - C \leq d(\gamma(s), \gamma(t))\leq C\abs{t - s} + C, 
    \end{align}
    for all $s, t\in I\subset \R$. 
\end{defn}

\subsection{Morse and strongly contracting (quasi-)geodesics}

In this section, we recall the definition of Morse (quasi-)geodesics and strongly-contracting quasi-geodesics. We then recall key properties of Morse (resp strongly-contracting) quasi-geodesics. For a more detailed background on properties of Morse (resp strongly-contracting) quasi-geodesics, we recommend \cite{C:Morse, CS:contracting} and \cite{ACHG:contraction_morse_divergence}.

\subsubsection{Morse (quasi-)geodesics}

\begin{defn}
    A function $M: \R_{\geq 1}\times \R_{\geq 0}\to \R_{\geq 0}$ is called a \emph{Morse gauge}, if it is non-decreasing and continuous in the second coordinate.
\end{defn}

\begin{defn}[Morseness]
    A quasi-geodesic $\gamma$ is called \emph{$M$-Morse} for some Morse gauge $M$ if every $C$-quasi-geodesic $\lambda$ with endpoints on $\gamma$ stays in the closed $M(C, C)$-neighbourhood of $\gamma$. A quasi-geodesic is called \emph{Morse} if it is $M$-Morse for some Morse gauge $M$.
\end{defn}

The following lemma is a well-known fact stating that Morse quasi-geodesics are at bounded Hausdorff distance from geodesics, proofs can be found for example in \cite{CS:contracting}. 

\begin{lemma}\label{lemma:morse-traced-by-geodesic}
    Let $M$ be a Morse gauge and $C\geq 1$ a constant. There exists a constant $D$ such that the following holds. Let $\lambda$ be an $M$-Morse $C$--quasi-geodesic. There exists a geodesic $\gamma$ such that the Hausdorff distance between $\lambda$ and $\gamma$ is at most $D$. 
\end{lemma}

\begin{proof}
    For quasi-geodesic segments, this is the statement of Lemma 2.5 (3) of \cite{CS:contracting}. For quasi-geodesic rays, this is proven in the proof of Lemma 2.10 of \cite{CS:contracting}. For bi-inifinite quasi-geodesics, the proof is analogous to the proof of Lemma 2.10 of \cite{CS:contracting}.
\end{proof}

\subsubsection{Strongly-contracting (quasi-)geodesics}

\begin{defn}[Strongly-contracting]
    Let $D\geq 0$ be a constant. We say that a geodesic $\gamma$ is \emph{$D$-contracting} if for every point $x\in X$, $\diam(\pi_\gamma(B_{d(x, \gamma)} (x)))\leq D$. A geodesic is called \emph{strongly-contracting} if it is $D$-contracting for some constant $D$. 
\end{defn}

As shown in \cite[Theorem 7.1]{ACHG:contraction_morse_divergence}, being strongly-contracting is equivalent to satisfying the bounded geodesic image property. 

\begin{lemma}[Direct consequence of \cite{ACHG:contraction_morse_divergence} Theorem 7.1]\label{lemma:equivalence_strong_contraction1}
    There exists a function $\Phi_{con}:\R_{\geq 0}\to \R_{\geq 0}$ such that for all quasi-geodesics $\gamma$, the following properties are $\Phi_{con}$--equivalent: 
    \begin{enumerate}
        \item \textbf{Bounded geodesic image property:} Any geodesic $\lambda$ with $d(\gamma, \lambda)\geq C$, satisfies $\diam(\pi_\gamma(\lambda))\leq C$. We say that $\gamma$ has \emph{$C$-bounded-geodesic image}.
        \item \textbf{Strong-contraction:} The geodesic $\gamma$ is $C$-contracting.
    \end{enumerate}
\end{lemma}

The following lemma states that strong contraction of geodesics behaves well under taking subsegments; a proof can be found in \cite[Theorem 1.1]{EZ:contracting}. 

\begin{lemma}[Strongly-contracting subsegments]\label{lemma:subsegments}
    There exists a function $\Phi_{sub} : \R_{\geq 0}\to \R_{\geq 0}$ such that every subgeodesic $\gamma'$ of a $C$-contracting geodesic $\gamma$ is $\Phi_{sub}(C)$-contracting.
\end{lemma}

Theorem 1.4 of \cite{ACHG:contraction_morse_divergence} implies that being strongly-contracting implies being Morse, which can be used to prove further results about strongly-contracting quasi-geodesics.

\begin{lemma}[Implication of Theorem 1.4 of \cite{ACHG:contraction_morse_divergence}]\label{lem:theorem1.4}
    Let $C\geq 0$ be a constant. There exists a Morse gauge $M$ only depending on $C$ such that any $C$--contracting quasi-geodesic is $M$--Morse.
\end{lemma}

Using the above result, we can state Lemma \ref{lemma:morse-traced-by-geodesic} in terms of strongly-contracting quasi-geodesics, which is the statement of the Lemma below.

\begin{lemma}[Lemma \ref{lemma:morse-traced-by-geodesic} for strong-contraction]\label{lemma:sc-traced-by-geodesic}
    Let $C\geq 1$ be a constant. There exists a constant $D$ such that the following holds. Let $\lambda$ be a $C$--contracting $C$--quasi-geodesic, then there exists a geodesic $\gamma$ such that the Hausdorff distance between $\lambda$ and $\gamma$ is at most $D$. 
\end{lemma}

Lemma \ref{lemma:subsegments} and Lemma \ref{lemma:sc-traced-by-geodesic} can be used to show that quasi-geodesics close to strongly-contracting geodesics are strongly-contracting.

\begin{lemma}\label{lemma:hausdorff_contraction}
    There exists a function $\Phi_{d} : \R_{\geq 0}\to \R_{\geq 0}$ such that the following hold. 
    \begin{enumerate}
        \item Every $C$--quasi-geodesic $\lambda$ with Hausdorff distance at most $C$ from a $C$-contracting $C$--quasi-geodesic $\gamma$ is $\Phi_{d}(C)$--contracting.
        \item Every $C$--quasi-geodesic segment $\lambda$ with endpoints contained in a $C$--neighbourhood of a $C$--contracting $C$--quasi-geodesic $\gamma$ is $\Phi_{d}(C)$--contracting.
    \end{enumerate}
\end{lemma}

\begin{proof}
    (1): Let $\lambda$ be a $C$-quasi-geodesic with Hausdorff distance at most $C$ from a $C$--contracting $C$--quasi-geodesic $\gamma$. In light of Lemma \ref{lemma:equivalence_strong_contraction1} it is enough to show that $\lambda$ has the $C''$--bounded geodesic image property, where $C''$ only depends on $C$. 

    Since $\gamma$ is $C$--contracting, it has the $C'$--bounded geodesic image property for some $C'$ only depending on $C$ by Lemma \ref{lemma:equivalence_strong_contraction1}.
    
    Let $x\in X$. We next show that $d_{\mathrm{Haus}}(\pi_{\gamma}(x), \pi_{\lambda}(x))$ is bounded by $3C + 3C'$. If $d(x, \lambda)\leq C+C'$, then $d(x, \gamma)\leq 2C + C'$ and hence $d_{\mathrm{Haus}}(\pi_{\gamma}(x), \pi_{\lambda}(x))\leq (C+C') + (2C+C') = 3C + 2C'$. If $d(x, \lambda) > C+C'$, let $a\in \pi_{\lambda}(x)$. Further, let $b$ be on $[x, a]$ at distance $C+C'$ from $a$. We have that $d([x, a]_{[x, b]}, \gamma)\geq C'$. Since $\gamma$ has $C'$-bounded geodesic image, $\diam(\pi_{\gamma}([x, a]_{[x, b]}))\leq C'$. For any $c\in \pi_{\gamma}(b)$ we have that $d(c, b)\leq 2C+C'$ and hence $d(c, a)\leq 3C+2C'$. Thus $d_{\mathrm{Haus}}(\pi_{\gamma'}(x), \pi_{\lambda}(x))\leq 3C+3C'$.

    We are now ready to show that $\lambda$ hast the $C''$ bounded geodesic image property for $C'' = 7C' + 6C$. Let $\eta$ be a geodesic with $d(\eta, \lambda)\geq C'+C$. We have that $d(\eta, \gamma)\geq C'$. Since $\gamma$ has $C'$-bounded geodesic image, $\diam(\pi_{\gamma}(\eta))\leq C'$. By the argument above, 
    $d_{\mathrm{Haus}}(\pi_{\gamma}(\eta), \pi_{\lambda}(\eta))\leq 3C+3C'$ and hence $\diam(\pi_{\lambda}(\eta))\leq C' +6C+6C' = C''$, implying that $\lambda$ indeed has $C''$ bounded geodesic image. Lemma \ref{lemma:equivalence_strong_contraction1} about the equivalence of strong contraction and bounded geodesic image concludes the proof.

    (2):  Let $\lambda$ be a $C$-quasi-geodesic segment whose endpoint lie in the $C$--neighbourhood of a $C$--contracting $C$--quasi-geodesic $\gamma$. By Lemma \ref{lemma:morse-traced-by-geodesic}, there exists a geodesic $\gamma'$ with Hausdorff distance at most $C'$ from $\gamma$, which is $C''$--contracting by (1). Further $C'$ and $C''$ depend only on $C$. By \cite[Lemma 2.1]{C:Morse} combined with Lemma \ref{lem:theorem1.4}, there exists a subsegment $\gamma''\subset\gamma'$ and a constant $H$ depending only on $C''$ (and hence depending only on $C$) such that $d_{\mathrm{Haus}}(\gamma'', \lambda)\leq H$. By  Lemma \ref{lemma:subsegments} about subgeodesics of strongly-contracting geodesics the geodesic $\gamma''$ is $C'''$--contracting, where $C'''$ depends only on $C$. Now (1) concludes the proof.
\end{proof}

In light of Lemma \ref{lemma:hausdorff_contraction}, we can strengthen Lemma \ref{lemma:sc-traced-by-geodesic} by being able to say that not only are $C$-contracting $C$-quasi-geodesics close to a geodesic, but that geodesic is also $C'$-contracting for some $C'$ only depending on $C$. This is summarized in the following lemma. 
\begin{lemma}\label{lemm:nearby-geodesic}
    There exists a function $\Phi_{geo}: \R_{\geq 0}\to \R_{\geq 0}$ such that every $C$-contracting $C$-quasi-geodesic $\lambda$ has Hausdorff distance at most $\Phi_{geo}(C)$ from a $\Phi_{geo}(C)$-contracting geodesic $\gamma$.
\end{lemma}

\subsection{Quadrangle contraction}

In this subsection we introduce the notion of quadrangle contraction. We then show that quadrangle contraction is equivalent to strong contraction. Being able to work with quadrangle contraction instead of strong contraction, in particular the notion of thin geodesics, greatly simplifies the proofs about the contraction space defined in the next Section. 

Let $\gamma$ be a geodesic segment. The \emph{midpoint} of $\gamma$, denoted by $m_\gamma$, is the point on $\gamma$ which is equidistant from both endpoints $\gamma^-$ and $\gamma^+$ of $\gamma$. For any pair of points $x, y$ we denote $m_{[x, y]}$ by $m_{xy}$.    

\begin{defn}
    An \emph{$n$-gon} $\mc G$ is an $n$-tuple $(\gamma_1, \gamma_2, \ldots, \gamma_n)$ of geodesic segments satisfying $\gamma_i^+ = \gamma_{i+1}^-$ for all $1\leq i \leq n$. Here $\gamma_{n+1}$ denotes $\gamma_1$. We call a 2-gon, 3-gon or 4-gon a bigon, triangle or quadrangle respectively. 
\end{defn}

We say that an $n$-gon $\mc G = (\gamma_1, \gamma_2, \ldots, \gamma_n)$ \emph{contains} a geodesic segment $\gamma$, if $\gamma$ is a subsegment of $\gamma_1$. 

\begin{defn}[$r$-thin] A geodesic segment $\gamma$ is \emph{$r$-thin} if every quadrangle $\mc Q = (\gamma_1, \gamma_2, \gamma_3, \gamma_4)$ containing $\gamma$ satisfies that $\gamma_2\cup \gamma_3\cup \gamma_4$ intersects the closed $r$-ball around $m_\gamma$. We say that a geodesic segment is \emph{thin} if it is $r$-thin for some $r$.
\end{defn}

Note that degenerate quadrangles (i.e. one or multiple sides are a point) are also quadrangles. Thus if we know a segment $\gamma$ is $r$-thin, this allows us to make statements about triangles and bigons the segment $\gamma$ is contained in.  

\begin{defn}[quadrangle-contracting]
    A geodesic $\gamma$ is $r$-quadrangle-contracting if every subsegment $\gamma'\subset \gamma$ of length at least $3r$ is $r$-thin. 
\end{defn}

Before we prove that quadrangle-contraction and strong contraction are equivalent, we prove the following technical lemma. The lemma highlights a property about closest point projection. Its main use is Corollary \ref{cor:closet_point_projection_quadrangles}, which we usually apply to quadrangles that contain contain $r$-thin geodesics.

\begin{lemma}\label{lemma:closet_point_projection_quadrangles}
    Let $\gamma$ be a geodesic. Let $x\in X$ be a point with closest point projection $x'\in \pi_{\gamma}(x)$. Let $y$ be a point on $\gamma$, then $d(y, [x, x'])\geq \frac{d(y, x')}{2}$.
\end{lemma}

\begin{proof}
    Let $z \in [x, x']$. Since $x'$ is the closest point projection of $x$ on $\gamma$ we have that $ d(z, x')\leq d(z, y)$. By the triangle inequality, $d(x', y)\leq d(x', z) + d(z, y)\leq 2d(z, y)$, which concludes the proof.
\end{proof}

\begin{cor}\label{cor:closet_point_projection_quadrangles}
     Let $\gamma$ and $\gamma'$ be geodesics. Let $x$ and $y$ be points on $\gamma'$ with closest point projections $x'\in \pi_{\gamma}(x)$ and $y'\in \pi_{\gamma}(y)$. Let $m$ be a point on $\gamma$ with $d(x, m) > 2r$ and $d(y, m) > 2r$. If $d([x', x]\cup\gamma_{[x, y]}'\cup [y, y'], m)\leq r$, then $d(m, \gamma'_{[x, y]})\leq r$.
\end{cor}

\begin{lemma}\label{lemma:full_strong_contraction_equivalence}
There exists a function $\Phi_Q:\R_{\geq 0}\to \R{\geq 0}$ such that for any geodesic $\gamma$, the following properties are $\Phi_Q$-equivalent.
\begin{enumerate}
    \item The geodesic $\gamma$ is $r$--quadrangle contracting.
    \item The geodesic $\gamma$ is $C$--contracting.
    \item The geodesic $\gamma$ has $D$--bounded geodesic image.
\end{enumerate}
\end{lemma}

\begin{proof}
    Lemma \ref{lemma:equivalence_strong_contraction1} implies that $(2)\iff (3)$.

    $(2)\implies (1): $ Assume $\gamma$ is $C$--contracting. Let $\gamma'\subset \gamma$ be a subsegment of length $\ell(\gamma')\geq 3C'+1$, where $C' = \Phi_{con}(\Phi_{sub}(C))$ is a constant such that every subsegment of $\gamma$ has $C'$--bounded geodesic image (see Lemmas \ref{lemma:subsegments} and \ref{lemma:equivalence_strong_contraction1}). It suffices to show that $\gamma'$ is $r = (4C'+1)$--thin. Let $\gamma''\subset\gamma'$ be a subsegment of length $3C'+1$ which contains $m_{\gamma'}$. Let $Q = (\gamma_1, \gamma_2, \gamma_3, \gamma_4)$ be a geodesic quadrangle containing $\gamma'$. The set $\pi_{\gamma''}(\gamma_2\cup \gamma_3\cup \gamma_4)$ contains both endpoints of $\gamma''$. Thus $\diam(\pi_{\gamma''}(\gamma_i))\geq \ell(\gamma'')/3 > C'$ for some $i\in\{2, 3, 4\}$. Since $\gamma''$ has $C'$--bounded geodesic image $d(\gamma_i, \gamma'') < C'$ and hence $d(\gamma_i, m_{\gamma'})<C' + \ell(\gamma'') \leq r$.

    $(1)\implies (3):$ Let $\lambda$ be a geodesic with $\diam(\pi_{\gamma}(\lambda)) > 4r+1$. Let $x, y$ be points on $\lambda$ such that there exist points $a\in \pi_\gamma(x), b\in \pi_\gamma(y)$ with $d(a, b)> 4r+1$. Let $\gamma' = \gamma_{[a, b]}$. Consider the geodesic quadrangle $\mc Q = (\gamma', [b , y], \lambda_{[y, x]}, [x, a])$. Using that $\gamma$ is $r$--quadrangle-contracting and using Corollary \ref{cor:closet_point_projection_quadrangles}, we get that $d(\lambda_{[y, x]}, m_{\gamma})\leq r$. Hence $d(\lambda, \gamma)\leq r$, which shows that $\gamma$ has $D = (4r+1)$--bounded geodesic image.
\end{proof}

\begin{defn}\label{def:phi}
    We denote by $\Phi$ the function which is the maximum of the previously mentioned functions $\Phi_d, \Phi_{geo}, \Phi_{con}, \Phi_{sub}$ and $\Phi_{Q}$.
\end{defn}

\subsection{Morse and strongly contracting Elements} 

In this section we recall the notions of Morse and strongly contracting elements and outline how they relate to one another.

\textbf{Notation:} In this section, $G$ denotes a group and $(X, d)$ a geodesic metric space on which $G$ acts. All actions in this paper are assumed to be actions by isometries.

An element $g\in G$ is called \emph{loxodromic} (with respect to the action of $g\acts X$), if the map $\Z\to X$ defined via $n\mapsto g^n x$ is a quasi-isometry for some (or equivalently all) $x\in X$.

If $G$ is a finitely generated group acting geometrically on $X$, then an element $g\in G$ is \emph{Morse} if the set $\{g^ix\}_{i\in \Z}$ is at bounded Hausdorff distance from a Morse geodesic line. Since Morseness is preserved under quasi-isometries, being Morse (unlike being loxodromic) is an intrinsic property of the element $g\in G$ and does not dependent on the specific action or space $X$. 

\begin{defn}[strongly-contracting axis]
    We say that (the axis of) an element $g\in G$ is $C$--contracting, if $g\acts X$ is loxodromic and $g$ stabilizes a bi-infinite $C$-contracting $C$-quasi-geodesic. We say that the axis of $g$ is strongly-contracting if it is $C$-contracting for some constant $C\geq 0$. 
\end{defn}

Observe that if an element $g\acts X$ is $C$-contracting, then Lemma \ref{lemm:nearby-geodesic} states that $\{g^ix\}_{i\in \Z}$ is in the $\Phi(C)$--neighbourhood of a $\Phi(C)$--contracting geodesic for some $x\in X$. 

Observe that (by Lemma \ref{lemm:nearby-geodesic} and Lemma \ref{lemma:hausdorff_contraction}) an element $g\acts X$ is strongly-contracting if and only if $\{g^ix\}_{i\in \Z}$ is at bounded Hausdorff distance from a strongly-contracting geodesic line. If the action of a finitely generated group $G$ on $X$ is geometric and $g$ is strongly-contracting, then $g$ is Morse.

\subsection{Actions on hyperbolic spaces}

In this section we recall some background on acylindrical hyperbolicity and WPD elements. For more details we recommend \cite{O:acylindrical}.

\textbf{Notation:} As in the previous section, $G$ denotes a group acting (by isometries) on a geodesic metric space $(X, d)$.

\begin{defn}[acylindrical action]
    Let $G$ be a group that acts on a hyperbolic space $X$. We say that $G$ acts \emph{acylindrically} on $X$ if for all $R>0$ there exists $n\in \N$ and $D>0$ such that for all $x, y\in X$ with $d(x, y)>D$ the following holds:
    \begin{align}
        \abs{\{g\in G \mid d(x, gx)<R, d(y, gy)<R\}} < n.
    \end{align}
\end{defn}

A finitely generated group $G$ is called \emph{acylindrically hyperbolic} if it admits a non-elementary acylindrical action. 

In our paper, we focus on slightly weaker forms of acylindrical actions. The first one is non-uniform acylindricity, where $\abs{\{g\in G \mid d(x, gx)<R, d(y, gy)<R\}}$ is bounded, but not necessarily uniformly bounded, this was introduced in \cite{Gen:cat-acylindrical}. The second one is an acylindrical action along a subset, introduced in \cite{S:hypemb}. There, $\abs{\{g\in G \mid d(x, gx)<R, d(y, gy)<R\}} < n$ but only for $x, y$ in a certain subset $A\subset X$ instead of the whole space.

\begin{defn}[Non-uniform acylindricity]\label{def:non-uniform}
    We say that a group $G$ acts non-uniformly acylindrically on a hyperbolic space $X$ if for all $R> 0$ there exists $D\geq 0 $ such that 
    \begin{align}\label{eq:non-uni}
        \abs{\{g\in G | d(x, gx) < R, d(y, gy) < R\}}< \infty,
    \end{align}
    for all $x, y\in X$ with $d(x, y)\geq D$.
\end{defn}

\begin{defn}\label{def:achyp_along_subsets}
    Let $G$ be a group which acts on a hyperbolic space $Y$ and let $A$ be a subset of $Y$. We say that $G$ acts \emph{acylindrically along $A$} if for all $R>0$ there exists $n\in \N$ and $D > 0$ such that for all $x, y\in A$ with $d(x, y)> D$ the following holds:
    \begin{align}\label{eq:acylindrical_along}
        \abs{\left \{ g\in G | d (x, gx) < R, d(y, gy) < R\right \}} < n.
    \end{align}
\end{defn}

In our paper, we will have actions which are acylindrical along the orbit of \emph{stable subgroups} of $G$.

\begin{defn}\label{def:stable_subgroup}
    Let $H$ be a finitely generated subgroup of a finitely generated group $G$. We say that $H$ is \emph{stable} if the following two properties hold
    \begin{enumerate}
        \item $H$ is undistorted (in other words, the natural embedding of $H$ in $G$ is a quasi-isometric embedding)
        \item There exists a Morse gauge $M$ such that any geodesic (in $G$) between any two points in $H$ is $M$--Morse in $G$.
    \end{enumerate}
\end{defn}

\subsubsection{WPD elements}

Another important notion are WPD elements.

\begin{defn}[WPD]\label{def:wpd}
    Let $G$ be a group that acts on a hyperbolic space $X$. A loxodromic element $g\in G$ is called WPD (weakly properly discontinuous) if for all $x\in X$ and $R>0$ there exists an integer $N\in \N$ such that 
    \begin{align}\label{eq:wpd}
        \abs{\left\{h\in G| d(x, h x) < R, d(g^Nx, hg^Nx)<R\right\}} < \infty.
    \end{align}
\end{defn}

An element $g\in G$ is called \emph{generalized loxodromic} if there is a hyperbolic space $X$ and an action of $G$ on $X$ for which $g$ is WPD. In \cite{S:hypemb}, it is shown that every generalized loxodromic element is a Morse element.

One focus of our paper is to construct actions where all generalized loxodromics act WPD. We call such actions universal WPD actions.

\begin{defn}\label{def:universal-wpd}
    Let $G$ be a group acting on a hyperbolic space $X$. We say that the action of $G$ on $X$ is a \emph{universal WPD action} if all generalized loxodromic elements are WPD.
\end{defn}

In Section \ref{sec:construction} we will construct a hyperbolic space with an action of a group $G$, which satisfies (under certain assumptions) that an element is WPD if and only if it is strongly contracting. Since all Morse elements are generalized loxodromics, such an action is a universal WPD action if and only if being a Morse element is equivalent to being a strongly contracting element. This leads us to the definition of \emph{Morse-dichotomous} and \emph{weakly Morse dichotomous} spaces, where this is precisely the case.

\begin{defn}[Morse-dichotomous]\label{def:morse-dich}
    We say that a geodesic metric space $X$ is \emph{Morse-dichotomous} if for every Morse gauge $M$ there exists a constant $C$ such that all $M$-Morse geodesics are $C$-contracting. We say that a geodesic metric space $X$ is \emph{weakly Morse-dichotomous} if every Morse geodesic is strongly-contracting. 
\end{defn}

Examples of Morse-dichotomous spaces include the following: 
\begin{itemize}
    \item CAT(0) spaces \cite{C:Morse_CAT(0)}
    \item Injective metric spaces \cite{SZ:injective}. Note that the mapping class group acts on an injective metric space \cite{HHP:injectivity}.
    \item Certain infinitely presented $C'(1/6)$--small cancellation groups \cite{Z:small_cancellation}.
\end{itemize}

\subsubsection{Quasi-isometrically embedded subspaces}

In \cite{BCKM:recog_space}, the notion of recognizing space for stable subgroups is introduced.

\begin{defn}\label{def:recognizing-space}
    Let $H\leq G$ be a stable subgroup and let $X$ be a hyperbolic space on which $G$ acts. If the orbit map from $H$ to $X$ is a quasi-isometry we say that $X$ is a \emph{recognizing space} for $H$. We further say that $X$ is a \emph{universal recognizing space} for $G$ if it is a recognizing space for all stable subgroups of $G$.
\end{defn}

In Section \ref{sec:acylindricity} we will show that under certain conditions on $X$ and the action of $G$ on $X$, the contraction space $\hat{X}$ is a universal recognizing space.

\section{Construction - the contraction space}\label{sec:construction}

In this section, we define the contraction space and show that it is hyperbolic. Intuitively, the construction works as follows. We want to ``collapse'' (i.e. add an interval of length one between) all pairs of points whose connecting geodesic is not ``strongly-contracting enough''. Since every geodesic segment is strongly-contracting (the diameter of any projection is at most the length of the segment) we need to specify what ``strongly-contracting enough'' means. It turns out that there are several different ways to define ``strongly-contracting enough'', which lead to different hyperbolic spaces which are not canonically isomorphic. In the following, we focus on a particular choice of making the notion ``strongly-contracting enough'' precise. In Section \ref{sec:alternative-definitions}, we discuss different choices of the definitions.

We introduce the notion of contraction gauge. 

\begin{defn}[contraction gauge]
    We call a function $K: \R_{\geq 0}\to \R_{\geq 0}\cup\{ \infty\}$ which is non-decreasing and satisfies $K(r)\geq 4r+1$ for all $r$ a \emph{contraction-gauge}. If $\infty$ is not in the image of $K$, we say that $K$ is \emph{full} otherwise we say that $K$ is \emph{partial}.
\end{defn}

Morally, a geodesic segment is ``strongly-contracting enough'' if for some $r\geq 0$, it is at least $K(r)$--long and $r$--thin. In the next section, we make the notion of ``strongly-contracting enough'' precise by defining what it means to be $K$-anti-contracting.\\
\textbf{Notation:} For the rest of the paper, unless noted otherwise, $K$ denotes a contraction gauge. 

\subsection{Construction}
We say that a pair of points $(x, y)$ is anti-contracting if none of its subsegments are ``strongly-contracting enough''.

\begin{defn}[Anti-contracting]\label{def:anti-contracting}
     We say that a pair of points $(x, y)\in X\times X$ is \emph{$K$-anti-contracting} if $d(x, y)\geq 1$ and for all $r\geq 0$ and all geodesics $\gamma$ from $x$ to $y$ the following holds. No subsegment $\gamma'$ of $\gamma$ of length $\ell(\gamma')\geq K(r)$ is $r$--thin. We denote the set of $K$--anti-contracting pairs by $\mc A_K$.
\end{defn}

Now we are ready to define the $K$-contraction space.

\begin{defn}[The $K$-contraction space]\label{def:K-contraction-space}
    The $K$-contraction space of $X$, denoted by $\hat{X}_K$ is constructed as follows; Start with the space $X$. Then, for every pair of $K$-anti-contracting points $(x, y)\in \mc A_K$ add an interval of length one and glue its endpoints to $x$ and $y$ respectively. 
\end{defn}

We denote the induced path metric in $\hat{X}_K$ by $\hat{d}$ if it is clear which contraction function was used and by $\hat{d}_K$ otherwise. Since $X$ is a geodesic metric space, so is $\hat{X}_K$. Similarly, we denote $D$-neighbourhoods in $\hat{X}_K$ by $\hat{\mc{N}}_D(\cdot)$, the diameter with respect to $\hat{d}$ by $\widehat{\diam}(\cdot)$ and so on. There is a natural inclusion $\iota_K:X\hookrightarrow \hat{X}_K$ and we identify $X$ with its image in $\hat{X}_K$. Observe that $d(x, y)\geq \hat{d}(x, y)$ for all $x, y\in X$. For any path $\gamma: I\to X$ we denote the composition $\iota_K\circ \gamma : I\to \hat{X}_K$ by $\hat{\gamma}$. While the images of $\gamma$ and $\hat\gamma$ are equal, the important distinction is that if $\gamma$ is a (quasi-)geodesic, $\hat{\gamma}$ need not be one. Further $\widehat{[x, y]}$ denotes the image of the geodesic $[x, y]$ in $\hat{X}_K$ and is not necessarily a geodesic (however, later on we will prove that $\widehat{[x, y]}$ is at bounded Hausdorff distance from a geodesic in $\hat{X}_K$). For any pair of points $x, y\in \hat{X}_K$, $[x, y]_K$ denotes a choice of geodesic (with respect to $\hat{d}$) from $x$ to $y$. Further, if we do not specify, then closest points, neighbourhoods, being a geodesic and so on is always considered to be with respect to the metric $d$.

\begin{rem}
    If $\hat{d}_K(x, y) > 1$ for some points $x, y\in X$, then there exists a constant $r\geq 0$, a geodesic $\gamma$ from $x$ to $y$ and a subsegment $\gamma'\subset \gamma$ such that $\gamma'$ is at least $K(r)$--ong and $r$--thin.
\end{rem}

The construction above associates a family $\{\hat{X}_K\}_K$ of what we will prove to be hyperbolic spaces to each geodesic metric space $\hat{X}_K$. Later on we will show that on mild assumptions on $K$ (i.e. $K(r)\geq 10r +1$ or $K$ being full) the $K$-contraction space satisfies a variety of desired properties. To not always have to specify the contraction gauge $K$, we want to distinguish one of these spaces and simply call it the contraction space. 

\begin{defn}[Contraction space]
    If $K(r) = 10r+1$, we denote $\hat{X}_K$ by $\hat{X}$ and call it the \emph{contraction space}.
\end{defn}

\subsection{Hyperbolicity} In this section, we first show that for any contraction gauge $K$, the $K$--contraction space $\hat{X}_K$ is hyperbolic. Secondly, we show that if $X$ contains a strongly-contracting geodesic ray and $K(r_0)\neq \infty$ for some large enough $r_0\geq 0$, then $(\hat{X}_K, \hat{d})$ is non-trivial.

\begin{prop}\label{prop:contraction_space_hyp_for_graphs}
    For any contraction gauge $K$, the $K$-contraction space $\hat{X}_K$ is $\delta$--hyperbolic, where $\delta$ is a constant neither depending on $K$ nor $X$. 
\end{prop}

To prove this, we use the following proposition of \cite{DDLS:veech}. It is a version of the guessing geodesic lemma from \cite{MS2013} which works in a more general setting. 

\begin{lemma}[Guessing geodesics, Proposition 2.2 of \cite{DDLS:veech}] \label{lemma:guessing_geodesics}Let $Y$ be a path-connected metric space, let $S\subset Y$ be an $R$-dense subset for some $R > 0$, and let $\delta\geq 0$ such that for all pairs $x, y\in Y$ there are rectifyable path-connected sets $\eta(x, y)\subset Y$ containing $x, y$ and satisfying:
\begin{enumerate}
    \item for all $x, y\in S$ with $d(x, y)\leq 3R$ we have $\diam(\eta(x, y))\leq \delta$,\label{guess_geo1}
    \item for all $x, y, z\in S$ we have $\eta (x, y)\subset \mc N_{\delta}(\eta(x, z)\cup \eta(z, y))$.\label{guess_geo2}
\end{enumerate}
Then there exists a constant $\delta'$ depending only on $\delta$ and $R$ such that $Y$ is $\delta'$-hyperbolic and the Hausdorff distance between $\eta(x, y)$ and any geodesic from $x$ to $y$ is at most $\delta'$.
\end{lemma}

\begin{figure}
    \centering
    \includegraphics{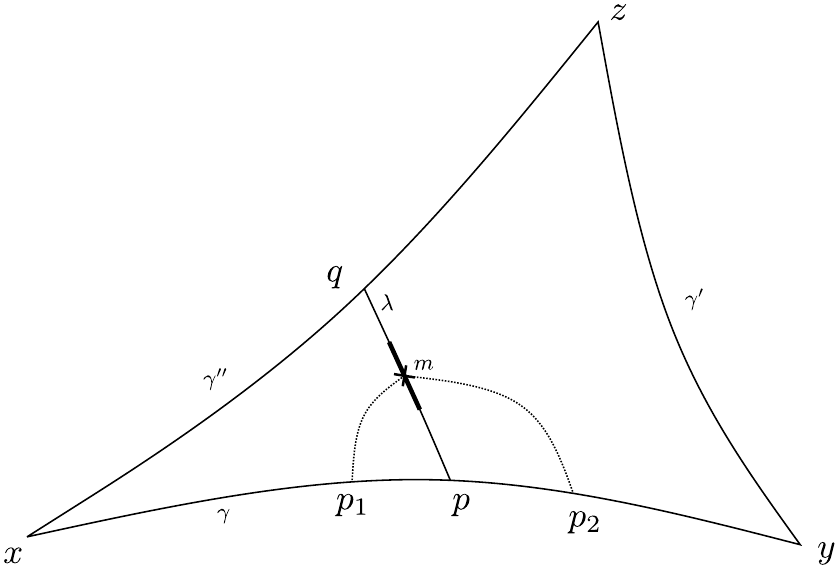}
    \caption{Geodesic triangles are thin}
    \label{fig:hyp-proof2}
\end{figure}

\begin{proof}[Proof of Proposition \ref{prop:contraction_space_hyp_for_graphs}]
    We use Lemma \ref{lemma:guessing_geodesics} about guessing geodesics where we choose $S = X$ and define $\eta(x, y) = \widehat{[x, y]}$ for all $x, y\in X$. The set $X$ is $\frac{1}{2}$-dense in $\hat{X}_K$.
    Clearly, $x, y\in \eta(x, y)$, further, $\ell(\eta(x, y))\leq d(x, y)$ so $\eta(x,y)$ are indeed rectifyable paths. It remains to prove \eqref{guess_geo1} and \eqref{guess_geo2}, which we will prove for $\delta=6$. 

    We first prove the following claim, which directly implies \eqref{guess_geo2}. 
    
    \begin{claim}
    If $(\gamma, \gamma', \gamma'')$ is a geodesic triangle in $X$, then $\gamma\subset \hat{\mc N}_1(\gamma'\cup\gamma'')$.     
    \end{claim}
    
    \textit{Proof of claim.} Denote the starting points of $\gamma, \gamma'$ and $\gamma''$ by $x, y$ and $z$ respectively and let $p$ be a point on $\gamma$. We will show that $\hat{d}(p, q)\leq 1$, where $q$ is a point on $\gamma'\cup \gamma''$ closest to $p$ with respect to the metric $d$. Without loss of generality $q$ lies on $\gamma''$. Assume that $\hat{d}(p, q) > 1$. Then the pair $(p, q)$ is not $K$--anti-contracting. In particular, there exists a constant $r\geq 0$, a geodesic $\lambda: I \to X$ from $p$ to $q$ and a subgeodesic $\lambda'\subset \lambda$ which is at least $K(r)$--long and $r$--thin. Let $m$ be the midpoint of $\lambda'$ as depicted in Figure \ref{fig:hyp-proof2}. Note that $d(m, p)\geq K(r)/2> 2r$ (and similarly $d(m, q) > 2r$). Considering the triangle $(\lambda, \gamma_{[q, x]}'' ,\gamma_{[x, p]})$ yields that there exists a point $p_1$ on $\gamma_{[q, x]}''\cup\gamma_{[x, p]}$ with $d(m, p_1)\leq r$. We have that $d(p, p_1)\leq d(p, m) + r < d(p, q)$. Hence the point $p_1$ cannot lie on $\gamma_{[q, x]}''$ (otherwise $q$ would not be a closest point to $p$) implying that $p_1\in \gamma_{[x, p]}$. A similar argument for the quadrangle $(\lambda, \gamma_{[q, z]}'', \gamma'_{[z, y]}, \gamma_{[y, p]})$ yields that there exists a point $p_2\in \gamma_{[y, p]}$ with $d(m, p_2)\leq r$. Since $p$ lies on a geodesic from $p_1$ to $p_2$, there exists $i\in \{1, 2\}$ such that $d(p, p_i)\leq r$. Hence $d(p, m)\leq 2r$, a contradiction. Thus $\hat{d}(p, q)\leq 1$. Since this holds for any vertex $p\in \gamma$, the statement follows. \hfill $\blacksquare$
    
    \eqref{guess_geo1}: Let $x, y\in X$ with $\hat{d}(x, y)\leq 3/2$. If $d(x, y)\leq 3/2$, then $\widehat{\diam}{(\eta(x, y))}\leq 3/2\leq \delta$. Otherwise, there exist a pair of $K$-anti-contracting points $(x', y')$ with $d(x, x')\leq 1/2$ and $d(y, y')\leq 1/2$. By the definition of anti-contracting, for any $z, z'\in [x', y']$ we have that $d(z, z')\leq 1$ or $(z, z')$ is $K$--anti-contracting as well. In particular, $\widehat{\diam}([x', y'])\leq 1$. The claim proven above applied to the triangles $([x, y], [y, x'], [x', x])$ and $([y, x'], [x', y'], [y', y])$ yields that $[x, y]\subset \hat{\mc N}_2([x, x']\cup [x', y']\cup[y', y])$. Hence $\widehat{\diam}([x, y])\leq 6$.
\end{proof}

\begin{rem}\label{rem:delta_0}
    We denote the constant $\delta'$ from Lemma \ref{lemma:guessing_geodesics} for $\delta = 6$ by $\delta_0$. With this notation, the $K$-contraction space is $\delta_0$ hyperbolic.
\end{rem}

\subsection{Quasi-geodesics in the contraction space} In this section we investigate which geodesics in $X$ map to parameterised quasi-geodesics in $\hat{X}_K$.

\begin{rem}\label{rem:on_geodesic}
    Instead of choosing $\eta(x, y) = \widehat{[x, y]}$, we could have chosen $\eta(x, y) = \hat{\gamma}$ for any geodesic $\gamma$ from $x$ to $y$ and the proof would have worked analogously. Hence, any geodesic $\gamma : I\to X$ from $x$ to $y$ is at Hausdorff distance at most $\delta_0$ from $[x, y]_K$. In particular, if $z$ lies on a geodesic from $x$ to $y$, then $\hat{d}(x, z)\leq \hat{d}(x, y) + \delta_0$. Moreover, $\widehat{\mathrm{diam}}(\gamma)\leq \hat{d}(x, y) + 2\delta_0$ for any geodesic $\gamma : I\to X$ from $x$ to $y$. Indeed, let $p, q\in \gamma$ for some $p, q$ with $d(x, p)\leq d(x, q)$. Applying the the statement above twice, we get that $\hat{d}(p, y)\leq \hat{d}(x, y) + \delta_0$ and $\hat{d}(p, q)\leq \hat{d}(p, y) + \delta_0$, implying that $\hat{d}(p, q)< \hat{d}(x, y) + 2 \delta_0$.
\end{rem}

\begin{lemma}\label{lemma:contracting_implies_quasi_geodesic} Let $K$ be a contraction-gauge and let $\gamma:  I \to X$ be a $C$-contracting geodesic. If $K(\Phi(C)^3)< \infty$, then $\hat{\gamma} : I\to \hat{X}_K$ is a $Q$-quasi-geodesic, where $Q = 4(K(\Phi(C)^3)+6\Phi(C) + 3)$.
\end{lemma}

Recall that $\Phi$ is defined in Definition \ref{def:phi}.

\begin{proof}
Let $\pi : \hat{X}_K\to X$ be the map that sends $x\in \hat{X}_K$ to (one of) its closest points in $X$. Observe the following; if $\hat{d}(x, y)\leq \frac{1}{4}$, then $\hat{d}(\pi(x), \pi(y))\leq 1$. Let $\tau : \hat{X}_K\to \gamma$ be a map that sends points $x\in X$ to (one of) their closest points in $\gamma$ with respect to the metric $d$. Further for $x\in \hat{X}_K - X$ define $\tau(x) = \tau(\pi(x))$. We will show that for all $x, y\in \hat{X}_K$,
\begin{align}\label{eq:quasi-g1}
    d(\tau(x), \tau(y))\leq Q\hat{d}(x, y) +Q,
\end{align} 
which, when applied to points on $\gamma$, directly implies that $\hat{\gamma}$ is a $Q$--quasi-geodesic. To show \eqref{eq:quasi-g1}, by the triangle inequality, it is enough to show that 
\begin{align}\label{eq:quasi-g2}
    d(\tau(x), \tau(y))\leq \frac{Q}{4},
\end{align} 
for all $x, y\in \hat{X}_K$ with $\hat{d}(x, y)\leq \frac{1}{4}$, or alternatively for all $x, y\in X$ with $\hat{d}(x, y)\leq 1$. We will proceed to show the latter. 

Let $x, y\in X$ be two points with $\hat{d}(x, y) \leq 1$. Define $x' = \tau(x)$ and $y' = \tau(y)$. We aim to show that $d(x', y')\leq Q/4$. Assume that $d(x', y') > Q/4$. Let $a'$ and $b'$ be points on $\gamma_{ [x', y']}$ at distance $2\Phi(C)+1$ from $x'$ and $y'$ respectively. Since $\gamma$ is $C$-contracting, it is $\Phi(C)$-quadrangle-contracting. In particular, considering the quadrangle $([x', y'], [y', y], [y, x], [x', x])$ yields points $a$ and $b$ on $[y', y]\cup[y, x]\cup[x, x']$ which are in the $\Phi(C)$-neighbourhood of $a'$ and $b'$ respectively. By Corollary \ref{cor:closet_point_projection_quadrangles}, $a$ and $b$ lie on $[x, y]$. By Lemma \ref{lemma:hausdorff_contraction}, $[x, y]_{[a, b]}$ is $\Phi^2(C)$-contracting and hence $\Phi^3(C)$-quadrangle-contracting. Consequently, $[x, y]_{[a, b]}$ is  $\Phi^3(C)$-thin. Further, $d(a, b) \geq d(x', y') - (6\Phi(C) +2 ) > K(\Phi^3(C))$, implying that $d(x, y) > 1$ and $(x, y)$ is not $K$--anti-contracting, a contradiction to $\hat{d}(x, y) \leq 1$. 
\end{proof}

The following corollary follows directly from Lemma \ref{lemma:contracting_implies_quasi_geodesic}

\begin{cor}\label{corollary:unbounded_if_strongly_contracting}
    If $X$ contains a strongly-contracting geodesic ray and $K$ is full, then $\hat{X}_K$ is unbounded.
\end{cor}

Next we want to prove Lemma \ref{lemma:quasi-geodesic_image_implies_strongly_contracting}, which is a partial converse of Lemma \ref{lemma:contracting_implies_quasi_geodesic}. Before we start with the proof we introduce the notion of separatedness, which will prove useful in the proof. 

\begin{defn}[separated]
    Let $x, y\in X$ be points. We say that $x$ and $y$ are $(r, K)$--separated for some $r\geq 0$ if there exists a geodesic $\gamma$ from $x$ to $y$ and a subgeodesic $\gamma'\subset \gamma$ which is at least $K(r)$--long and $r$--thin. If it is clear which contraction gauge we are using, we simply say that $x$ and $y$ are $r$--separated. We call the geodesic $\gamma$ an \emph{$r$--witness} of the separation of $x$ and $y$ and we call the point $m_{\gamma'}$ the \emph{pinch point}.
\end{defn}

Note that if $\hat{d}(x, y)>1$, then $x$ and $y$ are $r$--separated for some $r\geq 0$.

\begin{figure}
    \centering
    \includegraphics{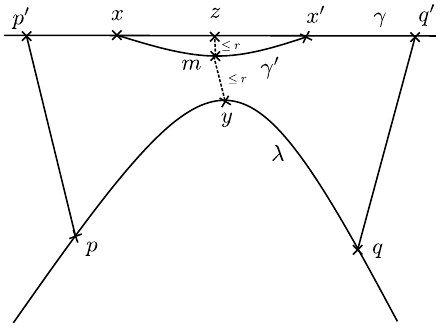}
    \caption{Quasi-geodesics are stongly contracting}
    \label{fig:quasi-geo-to-contr}
\end{figure}

\begin{lemma}\label{lemma:quasi-geodesic_image_implies_strongly_contracting}
Let $\gamma: I\to X$ be a geodesic. If $\hat{\gamma}$ is a $Q$-quasi-geodesic, then $\gamma$ is $\Phi(27Q^2)$-contracting. 
\end{lemma}

\begin{proof}
    Let $\gamma : I\to X$ be a geodesic such that $\hat{\gamma}$ is a $Q$-quasi-geodesic. We show that $\gamma$ has the $(27Q^2)$--bounded geodesic image property, and hence is $\Phi(27Q^2)$--contracting. 
    
    Let $\lambda : J\to X$ be another geodesic. Assume that there exist points $p, q$ on $\lambda$ whose closest point projections $p'\in \pi_\gamma(p)$ and $q'\in \pi_\gamma(q)$ onto $\gamma$ satisfy $d(p', q')\geq 27Q^2$. Since $\hat{\gamma}$ is a $Q$--quasi-geodesic, there exist points $x, x'$ on $\gamma_{[p', q']}$ with $\hat{d}(x, x') = 2$, $12 Q^2 \leq d(p', x)\leq d(p', x')$ and $12 Q^2\leq d(x', q')$. Since $\hat{\gamma}$ is a $Q$--quasi-geodesic, $d(x, x')\leq 3Q^2$. Hence the points $x$ and $x'$ are $r$--separated for some $r$ with $K(r)\leq 3Q^2$ (and thus $r\leq Q^2$). Let $\gamma'$ be an $r$--witness with pinch point $m$ of the separation of $x$ and $x'$. This is depicted in Figure \ref{fig:quasi-geo-to-contr}. Consider the quadrangle $\mc Q = (\gamma_{[p', x]}\circ\gamma'\circ\gamma_{[x', q']}, [q', q], \lambda_{[q, p]} ,[p, p'])$ and the bigon $\mc B = (\gamma', \gamma_{[x, x']})$. Since $m$ is the pinch point, there exists a point $y\in [q', q]\cup\lambda_{[q, p]} \cup[p, q]$ such that $d(y, m)\leq r$. Further, there exists a point $z\in \gamma_{[p, x']}$ with $d(z, m)\leq r$. Hence $d(y, z)\leq 2r$ but $d(z, \{p', q'\})\geq 4r$. By Corollary \ref{cor:closet_point_projection_quadrangles}, we can assume that $y\in \lambda_{[q, p]}$. 
    
    Hence $d(\lambda, \gamma)\leq 2r\leq 6Q^2$ implying that $\gamma$ has $(27Q^2)$-bounded geodesic image and hence is $\Phi(27Q^2)$--contracting by Lemma \ref{lemma:equivalence_strong_contraction1}.
\end{proof}

\subsection{Alternative definitions}\label{sec:alternative-definitions}

It turns out that in the definiton of the $K$-contraction space $\hat{X}_K$, there were some choices involved. Most prominently, how to define what ``not strongly-contracting enough'' should mean. In this section we explore other ways we could have defined the $K$--contraction space $\hat{X}_K$.  

Denote by $\mc B_K\subset X\times X$ the set of pairs $(x, y)$ with $d(x, y)\geq 1$ and such that for all geodesics $\gamma$ from $x$ to $y$ the following holds. No subsegment $\gamma'$ of $\gamma$ of length $\ell(\gamma')\geq K(r)$ is \textbf{$r$--quadrangle-contracting}. 

Note that the only difference between the definition of $\mc A_K$ and $\mc B_K$ is that we replaced ``$r$--thin'' by ``$r$--quadrangle-contracting''. Since any $r$--quadrangle-contracting geodesic is also $r$--thin, we have that $\mc A_K\subset \mc B_K$.

If we change the definition of anti-contracting, separated and witness as described below and change the definition of $\hat{X}_K$ accordingly, then all the results from Section \ref{sec:construction}, Section \ref{sec:acylindricity} and Section \ref{sec:diameter} still hold. The proofs work analogously.
\begin{itemize} 
    \item (anti-contracting) A pair $(x, y)\in X\times X$ is $K$-anti-contracting if and only if it is contained in $\mc B_K$.
    \item (separated) A pair of points $(x, y)$ is $(r, K)$--separated if there exists a geodesic $\gamma$ form $x$ to $y$ and a subgeodesic $\gamma'$ of $\gamma$ which is $r$--quadrangle contracting and at least $K(r)$--long.
    \item (witness) A geodesic $\gamma$ from $x$ to $y$ is a witness of a $(K, r)$-separation of the pair $(x, y)$ if there exists a subgeodesic $\gamma'\subset \gamma$ which is $r$--quadrangle contracting and at least $K(r)$--long.
\end{itemize}

We call the $K$-contraction space obtained from the definitions above the \emph{alternative $K$--contraction space} and denote it by $\tilde{X}_K$ (and its metric by $\tilde{d}$). The example below shows that while the $K$--contraction space and the alternative $K$--contraction space share quite a few properties, they are fundamentally different. 

\begin{ex}
    Let $X$ be the Cayley graph of the free product $G = \Z \ast \Z^2$ with respect to the standard generating set $S = \{a, b, c\}$, where $\Z = \langle a \rangle$ and $\Z^2 = \langle b, c\rangle$. Let $K$ be a full contraction gauge. We will show that the $K$--contraction space $\hat{X}_K$ is canonically quasi-isometric to the Bass-Serre tree $T$ of $X$, while the alternative $K$--contraction space $\tilde{X}_K$ is not canonically quasi-isometric to the Bass-Serre tree $T$. Consequently, the $K$-contraction space $\hat{X}_K$ is not canonically quasi-isometric to the alternative $K$-contraction space $\tilde{X}_K$.
    
    Let $Y$ be the Cayley graph of $\Z^2 = \langle b, c\rangle$. $X$ contains various copies of $Y$, which we will call \emph{sheets}. Consider the Bass-Serre tree $T$ of $X$. The canonical projection $\pi : X\to T$ sends every sheet $S$ of $X$ to a distinct vertex $v$ of $T$. We say that the distance between two sheets is the distance between their projections.

    \begin{obs}
        If two points $x, y\in X$ lie in the same sheet, then $(x, y)\in \mc A_K\subset \mc B_K$. 
    \end{obs}

    The observation above is a special case of the observation below.

    \begin{obs}
        If $\gamma: I\to X$ is a geodesic which is completely contained in a single sheet, then $\gamma$ is not $r$--thin for any $r < \ell(\gamma)/2$.
    \end{obs}
    
    Further note that any geodesic $\gamma$ whose midpoint $m_{\gamma}$ is contained in the interior of an edge $e$ of $X$ labelled by $a$ is 0-thin (any path from one endpoint of $\gamma$ to the other endpoint of $\gamma$ has to go through $m_{\gamma}$). This leads us to the following observation.
   
    \begin{obs}
        If two points $x, y\in X$ satisfy that their projections $\pi(x)$ and $\pi(y)$ in $T$ have distance at least $K(0)+1$, then $(x, y)\not\in \mc A_K$.
    \end{obs}
    Indeed, if the distance of $\pi(x)$ and $\pi(y)$ is at least $K(0)$, then there exists a point $z$ on $[x, y]$ which is in the interior of an edge labelled by $a$ and satisfies $d(x, z)\geq K(0)/2$ and $d(z, y)\geq K(0)/2$. Consequently, the subgeodesic $\gamma'$ of $[x ,y]$ of length $K(0)$ and midpoint $z$ is $0$--thin, implying that $(x, y)\not \in \mc A_K$.

    \begin{fact}
    The $K$--contraction space $\hat{X}_K$ is canonically quasi-isometric to the Bass-Serre tree $T$. In particular, any geodesic $\gamma: I\to X$ whose projection to $T$ is unbounded satisfies that $\hat{\gamma}$ is unbounded.
    \end{fact}
   
    This can be seen as follows. Extend the projection $\pi : X\to T$ to a projection $\hat{\pi} : \hat{X}_K\to T$, where $\hat{\pi}(x)$ for points $x\in \hat{X}_K - X$ is defined as $\pi(y)$, for some point $y\in X$ closest to $x$. The above observations show that $\hat{\pi}$ is a $(K(0)+2)$--quasi-isometry.  

    \begin{fact}
        The canonical projection $\tilde{\pi} : \tilde{X}_K\to T$ (defined analogously as $\hat{\pi}$) is not a quasi-isometry. In particular, there exists a geodesic $\gamma: I \to X$ whose projection $\pi(\gamma)$ to the Bass-Serre tree $T$ is unbounded but whose projection $\tilde{\gamma}$ to the alternative contraction space $\tilde{X}_K$ is bounded.
    \end{fact}

    Consider the geodesic $\gamma: I\to X$ (starting at any point, say $x_0$) and labelled by 
    \begin{align}
        b^{5^1}ab^{5^2}ab^{5^3}a \ldots.
    \end{align}
    \begin{obs}
        We have that $\pi(\gamma)$ is unbounded but $\tilde{\gamma}$ is bounded.
    \end{obs}
    
     The former is trivially true. We will prove the latter, by showing that $(x, y)\in \mc B_K$ for all $x, y\in \gamma$ with $d(x, y)\geq 1$. Assume that $\gamma'$ is a subgeodesic of $\gamma$ which is $r$--quadrangle-contracting and at least $K(r)$--long. By the definition of $\gamma$, there exists a subsegment $\gamma''$ of $\gamma$ of length at least $(3r+1)$ and which is completely contained in a single sheet. By the observation above, $\gamma''$ is not $r$--thin, a contradiction to $\gamma'$-being $r$--quadrangle-contracting. Since $\gamma$ is chosen in a way such that between any pair of points $x, y\in \gamma$ there is a unique geodesic, the argument above shows that $(x, y)\in \mc B_K$ for all $x, y\in \gamma$ with $d(x, y)\geq 1$. In other words, the diameter of $\tilde{\gamma}$ is $1$.
\end{ex}

We can generalize the construction of a contraction space even further. Namely, given a contraction gauge $K$ and a constant $r\geq 0$ we can choose a set of geodesics segments $\mc G_K(r)$ such that the following holds: 
\begin{itemize}
    \item If $\gamma\in \mc G_K(r)$, then $\gamma$ is $r$--thin and at least $K(r)$--long.
    \item If $\gamma$ is an at least $K(r)$--long $r$--quadrangle-contracting geodesic, then $\gamma\in \mc G_K(r)$.
\end{itemize}

With this, we can update the definition of anti-contracting, separated and witness as follows. 
\begin{itemize}
    \item (anti-contracting) A pair $(x, y)\in X\times X$ is $K$--anti-contracting if and only if for all all $r\geq 0$ and all geodesics $\gamma$ from $x$ to $y$, no subgeodesic $\gamma'$ of $\gamma$ is contained in $\mc G_K(r)$. We denote the set of $K$--anti-contracting pairs of points by $\mc C_K$.
    \item (separated) A pair of points $(x, y)$ is $(r, K)$--separated if there exists a geodesic $\gamma$ form $x$ to $y$ and a subgeodesic $\gamma'$ of $\gamma$ which is contained in $\mc G_K(r)$.
    \item (witness) A geodesic $\gamma$ from $x$ to $y$ is a witness of a $(K, r)$-separation of the pair $(x, y)$ if there exists a subgeodesic $\gamma'\subset \gamma$ which is contained in $\mc G_K(r)$.
\end{itemize}
We call the $K$-contraction space defined using these definitions the \emph{$(\mc G_K, K)$--contraction space}. As in the case of the alternative contraction space, the results in Section \ref{sec:construction} hold and the proofs are analogous. The results in Section \ref{sec:acylindricity} and Section \ref{sec:diameter} hold under the assumption that $\mc G_K(r)$ is $G$--invariant for all $r$, where $G$ is the group featuring in the statements. Again, the proofs work analogously.


\section{Group actions on the contraction space}\label{sec:acylindricity}
\textbf{Notation:} In this section, $G$ denotes a group acting by isometries on a geodesic metric space $(X, d)$.

Since the set of $K$--anti-contracting pairs of points is translation-invariant, the action $G\acts X$ naturally induces an action of $G \acts \hat{X}_K$. In this section we study the connection between the two actions and properties of $G\acts \hat{X}_K$. In particular we prove generalized versions of the Theorems \ref{thmintro:non-unifromv2}, \ref{thmintro:non-uniform-acylindricity}, \ref{thmintro:wpd-lox-contr-equivalence}, \ref{intro:stable} and \ref{intro:quasi-geodesics}.

\subsection{Acylindricity and WPD elements}
We start by studying WPD elements of $G \acts \hat{X}_K$ and then prove generalized versions of the Theorems \ref{thmintro:non-unifromv2}, \ref{thmintro:non-uniform-acylindricity} and \ref{thmintro:wpd-lox-contr-equivalence}. Most results follow directly from the construction of the $K$-contraction space. However, to prove Theorem \ref{thm:non-uniform-action-full} (which is a generalized version of Theorem \ref{thmintro:non-uniform-acylindricity}) we further need to show that closest points projections in the $K$--contraction space $\hat{X}_K$ are related to closest point projections in $X$ (Lemma \ref{lemma:closest_point_projections}).

\begin{prop}\label{prop:loxodromic_implies_wpd}
    Let $G$ be a group acting properly on $X$ and let $g\in G$ be an element such that $g\acts \hat{X}_K$ is loxodromic. Then $g\acts \hat{X}_K$ is WPD.
\end{prop}
\begin{proof}
    Let $g\in G$ be an element which acts loxodromically on $\hat{X}_K$. Since $X$ is $1$-dense in $\hat{X}_K$, it is enough to show that for every $x\in X$ and $R> 0$ there exists an integer $N\in \N$ such that
    \begin{align} \label{eq:lox_to_wpd}
        \abs{\{h\in G| d(x, hx)< R, d(g^Nx, hg^Nx) < R\}} < \infty.
    \end{align}

    Let $R>0$, let $x\in X$ and let $\tau = d(x, gx)$. Since $g\acts \hat{X}_K$ is loxodromic, there exists $c>0$ such that $\hat{d}(x, g^nx)\geq nc$ for all $n\in \N$. At the same time, $d(x, g^nx)\leq n\tau$. We will show that \eqref{eq:lox_to_wpd} holds for $N = (6\tau/c + 6 \delta_0 + 3R + 6)/c$, where we recall that $\delta_0$ is the hyperbolicity constant as defined in Remark \ref{rem:delta_0}.
    
    Let $M= \hat{d}(x, g^Nx)$. We have that $M\geq cN$ and hence $M\geq 6\tau/c + 6 \delta_0 + 3R + 6$. Further $d(x, g^Nx)\leq N\tau\leq M \tau/c$. Define $x_0 = x$. Further, for all $1\leq i \leq M/2$ choose a point $x_i$ on $[x, g^Nx]$ with $\hat{d}(x, x_i) = 2i$ and with $d(x, x_i)\geq d(x, x_{i-1})$. Observe that with this definition, $d(x, g^Nx)\geq \sum_{i=1}^{\floor{M/2}} d(x_{i-1}, x_{i})$.
    There exists an index $i_0$ with $M/6 -1 \leq i_0\leq M/3$ satisfying $d(x_{i_0}, x_{{i_0}+1})\leq 6\tau/c$. Indeed, otherwise $d(x, g^Nx)> (M/6) \cdot (6\tau/c)\geq M\tau/c$, a contradiction. Let $\gamma'$ be an $r$-witness of the separation of $x_{i_0}$ and $x_{i_0 +1}$ with pinch point $m$. Since $d(x_{i_0}, x_{{i_0}+1})\leq 6\tau/c$, we have that $r\leq 2\tau/c$.
    
    Let $\gamma$ be a geodesic from $x$ to $g^Nx$ of which $\gamma'$ is a subgeodesic. Let $h\in G$ be an element with $\hat{d}(x, hx)< R$ and $\hat{d}(g^Nx, hg^Nx)< R$. Consider the quadrangle $\mc Q = (\gamma, [g^Nx, hg^Nx], h\cdot [g^Nx, x], [hx, x])$. Since $m$ is a pinch point, there exists a point $y\in [g^Nx, hg^Nx]\cup  h\cdot [g^Nx, x]\cup [hx, x]$ with $d(y, m)\leq r \leq 2\tau/c$. We have that 
    \begin{align*}
        d(m, [hx, x])\geq \hat{d}(m, [hx, x])&\geq \hat{d}(m, x) - R - \delta_0 \geq \hat{d}(x_{i_0}, x) - R - 2\delta_0\\
        & \geq M/3 - R - 2\delta_0 - 2. 
    \end{align*}
    We use Remark \ref{rem:on_geodesic} for the second and third step. Since  $M > 2 \tau / c + 3R + 6 \delta_0 + 6$, the above calculation shows that $y\not\in [hx, x]$. One can show analogously that $y\not\in  [g^Nx, hg^Nx]$. Therefore, $d(m, h\cdot [g^Nx, x])\leq r$. Since the action of $G$ on $X$ is proper, only finitely many elements $h\in G$ satisfy this, which concludes the proof.
\end{proof}

\begin{prop}\label{prop:wpd_implies_strongly_contracting}
    Let $G$ be a finitely generated group acting geometrically on $X$. Let $g\in G$ be an element where $g\acts\hat{X}_K$ is WPD. Then $g\acts X$ is strongly-contracting.
\end{prop}

\begin{proof}
    Let $x\in X$. Since $g\acts \hat{X}_K$ is WPD, $g\acts X$ is Morse (see \cite[Theorem 1]{S:hypemb}). In particular, $\{g^ix\}_{i\in \Z}$ is at bounded Hausdorff distance (with respect to $d$) from a Morse geodesic line $\gamma : \R\to X$. Say $d_{Haus}(\{g^ix\}_{i\in \Z}, \gamma)\leq D$. Further, $g\acts \hat{X}_K$ is loxodromic and hence there exists $c>0$ such that $\hat{d}(x, g^nx)\geq cn$ for all $n\in \N$. 
    
    We next show that $\hat{\gamma}$ is a quasi-geodesic. Let $\tau = d(x, gx)$ and let $s, t\in \R$. There exists $i, j$ such that $d(\gamma(s), g^ix)\leq D$ and $d(\gamma(t), g^jx)\leq D$. Since, $d(g^ix, g^jx)\geq \abs{s - t} - 2D$, we have that $\abs{i-j} \geq \frac{\abs{s - t} - 2D}{\tau}$. Consequently, $\hat{d}(\hat{\gamma}(t), \hat{\gamma}(s))\geq \abs{i - j}c -2D \geq \frac{c\abs{s-t}}{\tau} - \frac{2Dc}{\tau}-2D$, implying that $\hat{\gamma}$ is indeed a quasi-geodesic. Lemma \ref{lemma:quasi-geodesic_image_implies_strongly_contracting} implies that $\gamma$ is strongly-contracting, which concludes the proof.
\end{proof}

\begin{prop}\label{prop:strongyl_contracting_implies_wpd}
    Let $g\in G$ be a element for which $g\acts X$ is $C$-contracting. If $K(\Phi^4(C)) < \infty$, then $g\acts \hat{X}_K$ is loxodromic.
\end{prop}

\begin{proof}
    Since $g\acts X$ is strongly-contracting, there exists $x\in X$ and a $\Phi(C)$-contracting geodesic $\gamma$ such that $\{g^ix\}_{i\in \Z}$ is in the $\Phi(C)$--neighbourhood of $\gamma$. Further the map $f : \Z \to X$ defined by $i\mapsto g^ix$ is a quasi-isometric embedding. We have to show that the map $\hat{f} : \Z \to \hat{X}_K$ defined by $i\mapsto g^ix$ is also a quasi-isometric embedding. 
    By Lemma \ref{lemma:contracting_implies_quasi_geodesic}, $\hat{\gamma}$ is a $Q$-quasi-geodesic for some $Q$.
    Let $i, j\in \Z$. There exist $p, q$ on $\gamma$ with $d(p, g^ix)\leq \Phi(C)$ and $d(q, g^jx)\leq \Phi(C)$. We have that 
    \begin{align*}
        \hat{d}(g^ix, g^jx)\geq \hat{d}(p, q)  - 2\Phi(C)\geq \frac{d(p,q)}{Q} - Q -2\Phi(C) \geq \frac{d(g^ix, g^jx) - 2\Phi(C)}{Q} - Q - 2\Phi(C).
    \end{align*}
    Since $f$ is a quasi-isometric embedding, the above inequalities imply that $\hat{d}(g^ix, g^jx)\geq \frac{\abs{ i - j }}{Q'} - Q'$ for some $Q'\geq 1$ and hence $\hat{f}$ is indeed a quasi-isometric embedding.
\end{proof}

\begin{thm}[Generalization of Theorem \ref{thmintro:wpd-lox-contr-equivalence}]\label{theorem:wpd_equivalence} Let $K$ be a full contraction gauge and let $G$ be a finitely generated group acting geometrically on $X$. Let $g\in G$ be an element. The following are equivalent;
\begin{enumerate}
    \item $g$ is loxodromic (w.r.t the action on $\hat{X}_K$),
    \item $g$ is WPD (w.r.t the action on $\hat{X}_K$),
    \item $g$ is strongly-contracting (w.r.t the action on $X$).
\end{enumerate}
\end{thm}
\begin{proof}
    Proposition \ref{prop:loxodromic_implies_wpd} states that $(1)\implies (2)$. Further, $(2)\implies (3)$ follows from Proposition \ref{prop:wpd_implies_strongly_contracting}. Lastly, $(3)\implies (1)$ follows from Proposition \ref{prop:strongyl_contracting_implies_wpd}.
\end{proof}

The following Theorem is a generalized version of the second part of Theorem \ref{thmintro:non-unifromv2}.

\begin{thm}\label{theorem:uniform_action}
    Let $K$ be a full contraction gauge and let $G$ be a finitely generated group acting geometrically on a weakly Morse-dichotomous space $X$. Then the action of $G$ on $\hat{X}_K$ is universally WPD.
\end{thm}

\begin{proof}
If $g\in G$ is a generalised loxodromic, then $g$ is Morse (see \cite[Theorem 1]{S:hypemb}) with respect to the action on $X$. Since $X$ is weakly Morse-dichotomous, $g$ is strongly-contracting with respect to the action on $X$. By Theorem \ref{theorem:wpd_equivalence}, $g$ is WPD with respect the action on $\hat{X}_K$.
\end{proof}

We show that if $G$ acts properly on $X$ the action on the $K$-contraction space $\hat{X}_K$ is non-uniformly acylindrical under the assumptions that either $K$ is partial or $K(r) \geq 10r+1$ for all $r\geq 0$. We start with proving that the action is non-uniformly acylindrical if $K$ is partial. 

\begin{thm}[Non-uniform acylindricity]
    Let $K$ be a partial contraction gauge and let $G$ be a group acting properly on $X$. Then the action of $G$ on $\hat{X}_K$ is non-uniformly acylindrical.
\end{thm}
\begin{proof}
    Since $X$ is $1$-dense in $\hat{X}_K$ it is enough to show \eqref{eq:non-uni} of Definition \ref{def:non-uniform} for points $x, y\in X$. 
    Let $r\geq 0$ be such that $K(r) = \infty$. Let $R> 0$ and let $x, y\in X$ be points with $\hat{d}(x, y)\geq 2k+2$ for $k = R+r+2\delta_0+2$. Recall that $\delta_0$ is the hyperbolicity constant defined in Remark \ref{rem:delta_0}. Let $z$ be a point on $[x, y]$ with $\hat{d}(x,z) = k$ and let $z'$ be a point on $[z, y]$ with $\hat{d}(z', y ) = k$. With this definition, $\hat{d}(z, z')\geq 2$. Let $\gamma'$ be an $r'$-witness of the separation of $z$ and $z'$ for some $r'$ and let $m$ be the pinch point. Observe that a geodesic can only be an $r'$ witness if $K(r')< \infty$. Since $K(r) = \infty$ and contraction gauges are non-decreasing, we know that $r'< r$.

    Let $\gamma$ be a geodesic from $x$ to $y$ containing $\gamma'$. Let $g\in G$ be such that $\hat{d}(x, gx)< R$ and $\hat{d}(y, gy)< R$. It remains to show that there are at most finitely many such elements $g$. Consider the quadrangle $\mc Q = (\gamma, [y, gy], g\cdot [y, x], [gx, x])$. Since $m$ is the pinch point, there exists a point $p\in [y, gy]\cup g\cdot [y, x]\cup [gx, x]$ with $d(p, m)\leq r'$. 
    Since 
    \begin{align*}
        d(m, [x, gx])\geq \hat{d}(m, z) - \widehat{\diam}([x, gx])\geq k - R - 2\delta_0 > r \geq r',
    \end{align*}
    we have that $p\not \in [x, gx]$. Note that $\widehat{\diam}([x, gx])\leq R + 2\delta_0$ by Remark \ref{rem:on_geodesic}. Similarly, $p\not\in [y, gy]$ and hence $p\in g\cdot[y, x]$. Since the action of $G$ on $X$ is proper, there are only finitely many $g\in G$ with $d(m, g\cdot [y, x])\leq r$. This concludes the proof.
\end{proof}

To prove non-uniform acylindricity for full contraction gauges, we need be able to control closest point projections in $\hat{X}_K$. More precisely, we want to show that if we take the closest point projections (onto a geodesic) with respect to $d$ or $\hat{d}$, then the resulting points are close with respect to $\hat{d}$. We show this in the following lemma.  

\begin{figure}
    \centering
    \includegraphics[width=.95\textwidth]{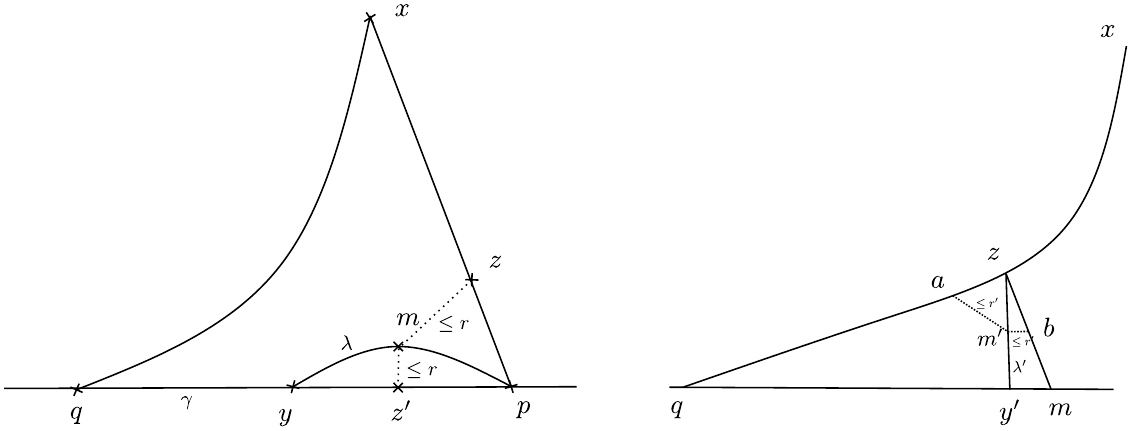}
    \caption{Closest point projections in $\hat{X}_K$}
    \label{fig:closest point proj}
\end{figure}

\begin{lemma}[Closest point projections]\label{lemma:closest_point_projections} Let $x\in X$ be a point, $\gamma: I \to X$ be a geodesic and $p$ (resp $q$) be closest points to $x$ on $\gamma$ with respect to the metric $d$ (resp $\hat{d}$). If $K(r) \geq 10r+1$ for all $r\geq0$, then $\hat{d}(p, q) < 17$.
\end{lemma}

\begin{proof}
    Assume that $\hat{d}(p, q)\geq 17$. Let $y$ be the point on $\gamma_{[q, p]}$ closest to $p$ which satisfies $\hat{d}(y, p)\geq 3/2$. Let $\lambda$ be an $r$-witness with pinch point $m$ of the separation of $y$ and $p$. This is depicted in Figure \ref{fig:closest point proj}. Consider the triangle $\mc Q = (\gamma_{[q, y]}\circ \lambda, [p, x], [x, q])$. Since $m$ is the pinch point of $\lambda$, there exists $z\in [p, x]\cup[x, q]$ with $d(m, z)\leq r$. We show that $z\in[x, q]$. Assume otherwise. Consider the bigon $\mc B = (\lambda, \gamma_{[p, y]})$. There exists a point $z'\in \gamma_{[p, y]}$ with $d(z', m)\leq r$. Hence $d(z, z')\leq 2r$. Lemma \ref{lemma:closet_point_projection_quadrangles} shows that $d(z', p)\leq 4r$ and hence $d(m, p) \leq 5r$. This is a contradiction to $d(m, p) \geq K(r)/2 > 5r$. So indeed $z\in [x, q]$. 

    For the rest of the proof we assume that $z$ is the point on $[x, q]$ closest to $m$. Let $y'$ be the point on $\gamma_{[q, y]}\circ \lambda$ closest to $z$. Since $d(z, m)\leq r$ we know that $d(m, y')\leq 2r$ and hence $y'$ lies on $\lambda$. In the proof of Proposition \ref{prop:contraction_space_hyp_for_graphs} it is shown that if $\hat{d}(a, b)\leq 3/2$ for points $a, b\in X$ and $\eta$ is a geodesic from $a$ to $b$, then $\widehat{\diam}(\eta)\leq \delta = 6$. Applying this to $a= y$ and $b = p$ yields that $\hat{d}(y, y')\leq 6$. Consequently, $\hat{d}(z, y)\leq \hat{d}(z, y') + 6$. Using the triangle inequality, we get that
    \begin{align*}
        17 - 3/2\leq \hat{d}(p, q) - 3/2\leq\hat{d}(q, y)\leq \hat{d}(q, z) + \hat{d}(z, y)\leq 2\hat{d}(z, y) \leq 2\hat{d}(z, y') + 12,
    \end{align*}
    where the fourth inequality holds since $q$ is the closest point on $\gamma$ to $x$ with respect to $\hat{d}$ (and hence the closest point to $z$ on $\gamma$ with respect to $\hat{d}$). Hence $\hat{d}(z ,y') > 1$. In particular, there exists an $r'$-witness $\lambda'$ with pinch point $m'$ of the separation of $y'$ and $z$. Now consider the triangle $(\gamma_{[q, y]}\circ\lambda_{[y, y']} ,\lambda', [x, q]_{[z, q]})$ since $m'$ is a pinch point, there exists a point $a\in\gamma_{[q, y]}\circ\lambda_{[y, y']}\cup [x, q]_{[z, q]}$ with $d(a, m')\leq r'$. Since $y'$ is the closest point on $\gamma_{[q, y]}\circ\lambda$ to $z$, $d(m', \gamma_{[q, y]}\circ\lambda)\geq K(r')/2 > r'$ and hence $a$ lies on $[x, q]_{[z, q]}$. Next consider the triangle $(\lambda', [z, m] ,\lambda_{[m, y']})$. Again since $m'$ is a pinch point, there exists $b\in [z, m] \cup\lambda_{[m, y']} $ with $d(b, m')\leq r'$. Analogously to the other triangle, $b\not\in \lambda_{[m, y']}$ and hence $b\in [z, m]$. Using the triangle inequality, we get that $d(z, b) \geq d(z, m') - d(m', b) \geq K(r')/2-r' > 4r'$, while $d(a, b)\leq 2r'$. Hence $d(a, m) <d(z, m)$ contradicting that $z$ is the point on $[x, q]$ closest to $m$ and thus yielding the desired contradiction to $\hat{d}(p, q) \geq 17$. 
\end{proof}

Next we prove the following lemma, which is a technical lemma about quadrangles with small sides in the contraction space. We will need it both for the proof of Theorem \ref{thmintro:non-uniform-acylindricity} and Proposition \ref{lemma:loxodromic_implies_strongly_contracting}, which shows that elements acting loxodromic on the contraction space $\hat{X}_K$ act strongly-contracting on $X$.

\begin{lemma}\label{lemma:quadrangle-estimates}
    Let $K$ be a full contraction gauge with $K(r)\geq 10r+1$ for all $r\geq 0$. For all $R>0$, there exists $D>0$ such that the following holds. Let $x, y\in X$ be points with $\hat{d}(x, y)\geq D = 2k+2$. Let $x', y'$ be points on $[x, y]$ such that $\hat{d}(x, x') = k$ and $\hat{d}(y, y') = k$. Let $r_0$ be a constant such that $x'$ and $y'$ are $r_0$-separated, we call $r_0$ a \emph{middle separator of $x$ and $y$}. Let $x'', y''$ be points with $\hat{d}(x, x'')<R$ and $\hat{d}(y, y'') < R$. Then $d([x, y], [x'', y''])\leq r_0$.
\end{lemma}

\begin{figure}
    \centering
    \includegraphics{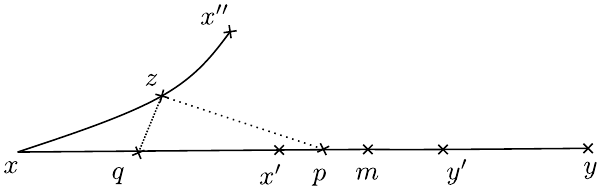}
    \caption{Quadrangle estimates}
    \label{fig:non-uni}
\end{figure}

\begin{proof}
    Let $R > 0$ we show that the statement holds for $k = 2R + 3\delta_0 + 18 $. Recall that $\delta_0$ is the hyperbolicity constant defined in Remark \ref{rem:delta_0}. 
    By the triangle inequality $\hat{d}(x', y') \geq 2$, so there does indeed exists $r_0$ such that $x'$ and $y'$ are $r_0$-separated. Let $\lambda$ be an $r_0$-witness with pinch point $m$ of the separation of $x'$ and $y'$. 
    
    Consider the quadrangle $\mc Q = ([x, x']\circ\lambda\circ[y', y], [y, y''], [y'', x''], [x'', x])$. Since $m$ is a pinch point, there exists $z\in [y, y'']\cup [y'', x'']\cup [x'', x]$ with $d(z, m)\leq r_0$. Assume that $z\in [x'', x]$, this is depicted in Figure \ref{fig:non-uni}. Let $p$ and $q$ be the closest point projections of $z$ onto $[x, x']\circ\lambda\circ[y', y]$ with respect to the metrics $d$ and $\hat{d}$. We have that $d(z, p)\leq d(z ,m)\leq r_0$ and hence $d(p, m)\leq 2r_0$, implying that $p\in \lambda$ and hence $\hat{d}(x, p)\geq \hat{d}(x,x') - \delta_0 = k - \delta_0$, by Remark \ref{rem:on_geodesic}. 
    
    Next we study $q$. Since $\hat{d}(x, x'') < R$ and $z$ lies on $[x'', x]$ we have that $\hat{d}(z, x)\leq R + \delta_0$ (see Remark \ref{rem:on_geodesic}). Consequently $\hat{d}(q, x)\leq \hat{d}(x, z) + \hat{d}(z, q)\leq 2\hat{d}(z, x) \leq 2R + 2\delta_0$. But now $\hat{d}(p, q)\geq \hat{d}(x, p) - \hat{d}(x, q)\geq k - 2R - 3\delta_0> 17$, a contradiction to Lemma \ref{lemma:closest_point_projections}, implying that $z\not\in [x'', x]$. Analogously we can show that $z\not\in [y, y'']$. Consequently, $z\in [y'', x'']$, which concludes the proof.
\end{proof}

\begin{thm}[Non-uniform acylindricity, generalization of Theorem \ref{thmintro:non-uniform-acylindricity}]\label{thm:non-uniform-action-full}
    Let $K$ be a full contraction gauge with $K(r) \geq 10r +1$ and let $G$ be a group acting properly on $X$. Then the action of $G$ on $\hat{X}_K$ is non-uniformly acylindrical. Furthermore, if there exists an element $g\in G$ whose axis in $X$ is strongly contracting, then $G\acts \hat{X}$ has unbounded orbits.
\end{thm}

\begin{proof}
    The furthermore part follows from Lemma \ref{lemma:contracting_implies_quasi_geodesic}. Since $X$ is $1$-dense in $\hat{X}_K$ it is enough to show \eqref{eq:non-uni} of Definition \ref{def:non-uniform} for points $x, y\in X$. Let $R > 0$ and let $D>0$ be as in Lemma \ref{lemma:quadrangle-estimates}. 
    
    Let $g\in G$ be an element such that $\hat{d}(x, gx) < R$ and $\hat{d}(y, gy) < R$. By Lemma \ref{lemma:quadrangle-estimates}, there exists $r_0\geq 0$ (not depending on $g$) such that $d([gx, gy], [x,y])\leq r_0$. Due to the properness of the action of $G$ on $X$, there are only finitely many elements $g\in G$ with $d([gx, gy], [x, y])\leq r_0$, which concludes the proof.
\end{proof}

The following proposition is a strengthening of Proposition \ref{prop:wpd_implies_strongly_contracting}. We will use it in the next section.

\begin{prop}\label{lemma:loxodromic_implies_strongly_contracting}
    Let $K$ be a full contraction gauge with $K(r)\geq 10r+1$ for all $r\geq 0$. Let $G$ be a group acting on $X$. If an element $g\in G$ acts loxodromically on $\hat{X}_K$, then $g$ acts strongly-contracting on $X$. 
\end{prop}
\begin{proof}
    Let $x\in X$ and let $g\in G$ be an element acting loxodromically on $\hat{X}_K$. In other words, the map $i\mapsto g^ix$ from $\Z\to (\hat{X}_K, \hat{d})$ (and consequently from $\Z$ to $(X, d)$) is a quasi-isometric embedding. In particular, there exists a constant $Q$ such that the path $\gamma_k : \R \to \hat{X}_K$
    \begin{align*}
        \gamma_k = \cdots \circ [g^{ik}x, g^{(i+1)k}]\circ [g^{(i+1)k}x, g^{(i+2)k}x]\circ \cdots
    \end{align*}
    is a $Q$-quasi-geodesic for all $k\in \N$. Since $\hat{X}_K$ is hyperbolic, there exists a constant $\delta$ such that the Hausdorff distance between a geodesic and a $Q$-quasi-geodesic with the same endpoints is at most $\delta$. Let $R = \delta + \delta_0+1$. With this, we have that $[x, g^Nx]$ and $(\gamma_k)_{[x, g^Nx]}$ have Hausdorff less than $R$ (with respect to $\hat{d}$) for all $N, k$ such that $N$ is divisible by $k$.

    Let $D$ be the constant from Lemma \ref{lemma:quadrangle-estimates}, when applying it to $R$. Let $k$ be an integer such that $\hat{d}(x, g^kx)\geq D$. Let $D' = d(x, g^kx)$.

    For the moment, fix $N$ such that $N$ is divisible by $k$. Denote by $y_i$ a point on $[x, g^Nx]$ such that $\hat{d}(y_i, g^{ik}x)\leq R$. Let $r_0$ be a middle separator of $x$ and $g^kx$. By Lemma \ref{lemma:quadrangle-estimates}, there exists $m\in [x, g^kx]$ and for all $i\in \Z$, there exists $z_i\in [x, g^Nx]_{[y_i, y_{i+1}]}$ such that $d(z_i, g^{ik}m)\leq r_0$. In particular, $d(z_i, g^{ik}x)\leq D' + r_0$. Implying that $d(z_i, z_{i+1})\leq 3D' + 2r_0$ and consequently that the Hausdorff distance between $(\gamma_k)_{[x, g^Nx]}$ and $[x, g^Nx]$ is bounded by $6D' + 4r_0$. Furthermore, we have that $\hat{d}(z_i, z_{j})\geq \hat{d}(g^{ik}x, g^{jk}x)  - 2D' - 2r_0$. In particular, $\widehat{[x, g^Nx]}$ is a $C$-quasi-geodesic for some constant $C$ not depending on $N$.

    Next we consider the sequence of geodesics $([g^{-ik}x, g^{ik}x])_{i\in \N}$. By Arzela-Ascoli, there exists a subsequence converging to a geodesic $\lambda$. Above we have shown that $\widehat{[g^{-ik}x, g^{ik}x]}$ is a $C$-quasi-geodesic for all $i$ and hence $\lambda$ is also a $C$-quasi-geodesic. Lemma \ref{lemma:quasi-geodesic_image_implies_strongly_contracting} implies that $\lambda$ is strongly-contracting. Above we have shown that the Hausdorff distance (with respect to $d$) between $[g^{-ik}x, g^{ik}x]$ and $(\gamma_k)_{[g^{-ik}x, g^{ik}x]}$ is uniformly bounded. Consequently, the Hausdorff distance of $\lambda$ and $\gamma_k$ is bounded. Hence $g$ is indeed strongly contracting.
\end{proof}

\subsection{Stable subgroups and universal recognizing spaces}

In this section we prove generalized versions of Theorem \ref{intro:stable} and Theorem \ref{intro:quasi-geodesics}. In particular, we show that if the action of a finitely generated group $G$ on $X$ is geometric and $X$ is Morse-dichotomous, then stable subgroups of $G$ embed quasi-isometrically into the contraction space (implying that the contraction space is a universal recognizing space) and the action along their orbit is acylindrical.

We start with proving a generalized version of Theorem \ref{intro:quasi-geodesics}.

\begin{thm}[Generalization of Theorem \ref{intro:quasi-geodesics}]\label{intro:quasi-geodesics-gen}
    Let $K$ be a full contraction gauge with $K(r)\geq 10r+1$ for all $r\geq 0$. Let $X$ be a geodesic metric space. Let $C\geq 0$ be a constant and let $Y\subset X$ be a subset such that $[x, y]$ is $C$-contracting for all $x, y\in Y$. Then the inclusion $(Y, d)\hookrightarrow (\hat{X}, \hat{d})$ is a quasi-isometric embedding. Moreover, if $G$ is a group acting properly on $X$, then the action $G\acts \hat{X}$ is acylindrical along $Y$.  
\end{thm}
\begin{proof}
    Using Lemma \ref{lemma:contracting_implies_quasi_geodesic}, we get a constant $Q$ such that $\widehat{[x, y]}$ is a $Q$-quasi-geodesics for all $x, y\in Y$. In particular, the inclusion $(Y, d)\hookrightarrow (\hat{X}, \hat{d})$ is a quasi-isometric embedding. 

    By Lemma \ref{lemma:full_strong_contraction_equivalence}, we get a constant $r$ such that $[x, y]$ is $r$--quadrangle-contracting for all $x, y\in Y$.

    It remains to show that $G$ acts acylindrically along $Y$. Let $R> 0$ define $D = 2R + 6r + 2\delta_0 + 1$. We will show that for all $x, y \in Y$ with $\hat{d}(x, y)\geq D$, we have that 
    \begin{align*}
        \abs{\{g\in G | \hat{d}(x, gx) < R, \hat{d}(y, gy) < R\}} < n,
    \end{align*}
    for some constant $n$ not depending on $x$ and $y$. Let $x, y\in Y$ with $d(x, y) > D$ and let $g\in G$ such that $\hat{d}(x, gx)<R$ and $\hat{d}(y, gy) < R$.  We will show that $d(x, gx) < D'$ for $D' = Q(D + R + r + 2Q)+r$. Then the properness of the action of $G$ on $X$ concludes the proof. 

    Let $m\in [x, y]$ be a point such that $\hat{d}(x, m) = D/2$. In particular, $\hat{d}(m, y)\geq D/2$, $d(x, m)\geq \hat{d}(x , m) \geq 3r$ and $d(y, m) \geq 3r$. Consider the quadrangle $\mc Q = ([x, y],[y, gy] ,g\cdot[y, x], [gx, x])$. Since $[x, y]$ is $r$--quadrangle-contracting (and $d(m, \{x, y\})\geq 3r$), there exists a point $p\in [y, gy]\cup g\cdot[y, x]\cup [gx ,x]$ such that $d(p, m)\leq r$. By Remark \ref{rem:on_geodesic} and the triangle inequality, for any point $q\in [x, gx]$ (or $q\in [y, gy]$) we have that $\hat{d}(m, q)\geq D/2 - R - \delta_0 > r$ and hence $d(m, q)> r$. Thus, $p\in g\cdot [y, x]$. Furthermore, by the triangle inequality, $\hat{d}(gx, p)\leq D/2 + R + r$. Since $\widehat{[y, x]}$ (and hence $g\cdot \widehat{[y, x]}$) is a $Q$-quasi-geodesic, we have that $d(gx, p)\leq Q( D/2 + R + r) + Q^2$ and $d(x, m)\leq QD/2 + Q^2$. Recall that $D' = Q(D + R + r + 2Q)+r$. By the triangle inequality, $d(gx, x)\leq D'$. 
\end{proof}

\begin{thm}[Generalization of Theorem \ref{intro:stable}]
    Let $K$ be a full contraction gauge with $K(r)\geq 10r+1$. Let $G$ be a finitely generated group acting geometrically on a Morse-dichotomous space $X$ and let $H\subset G$ be a stable subgroup. Then $H$ quasi-isometrically embeds into $\hat{X}_K$ and the action of $G$ on $\hat{X}_K$ is acylindrical along the orbit of $H$. In particular, $\hat{X}_K$ is a universal recognition space of $G$.
\end{thm}

\begin{proof}
Pick a basepoint $x_0\in X$. Since $H$ is a stable subgroup of $G$ (and hence undistorted in $G$) and $G$ acts geometrically on $X$, the map $\phi : H\to X$ defined by $\phi(h) = hx_0$ is a quasi-isometric embedding. Define $Y = \phi(H)$. Moreover, there exists a Morse gauge $M$ such that $[x, y]$ is $M$-Morse for all elements $x, y\in Y$. Since $X$ is Morse-dichotomous, there exist constants $C$ and $r$ such that $[x, y]$ is $C$-contracting and $r$-quadrangle-contracting for all $x, y\in Y$.

Theorem \ref{intro:quasi-geodesics-gen} shows that $(Y, d)\hookrightarrow (\hat{X}_K, \hat{d})$ is a quasi-isometric embedding and that $G$ acts acylindrical along $Y$. 
\end{proof}

\section{The diameter dichotomy}\label{sec:diameter}

In this section we prove Theorem \ref{thmintro:diameter_dich}. More precisely, we show that if $\mathrm{Isom}(X)$ acts coboundedly on $X$, then the $K$--contraction space $\hat{X}_K$ is either unbounded or has uniformly bounded diameter only depending on the \emph{density} of the action of $\mathrm{Isom}(X)$ on $X$.  We further show that in this case, the $K$-contraction space $\hat{X}_K$ is unbounded if and only if $X$ contains a strongly-contracting ray.

\begin{defn}[density]
    Let $G$ be a group acting coboundedly on a metric space $X$. We say that the density of $x\in X$, denoted by $\rho_x$ is the Hausdorff distance between $G\cdot x$ and $X$. We say that the density of the action is $\rho = \inf_{x\in X}\{\rho_x\}$.
\end{defn}

\begin{prop}\label{prop:diameter_dichotomy}
If $\mathrm{Isom}(X)$ acts coboundedly on $X$ with density $\rho$, then there exists a constant $\Delta = \Delta(\rho) >0$ depending only on $\rho$ such that for all contraction gauges $K$, one of the following holds: 
\begin{itemize}
    \item $\widehat{diam}_K(\hat{X}_K)\leq \Delta$,
    \item $\hat{X}_K$ is unbounded.
\end{itemize}
\end{prop}
\begin{figure}
    \centering
    \includegraphics{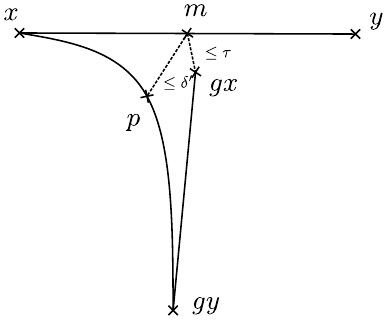}
    \caption{Proof of diameter dichotomy}
    \label{fig:diam}
\end{figure}

\begin{proof}
Let $x\in X$ such that $\rho_x \leq \rho + 1$. We show that the Proposition holds for $\Delta = 4\delta_0 + 2\rho + 9$, where $\delta_0$ is the hyperbolicity constant defined in Remark \ref{rem:delta_0}. Let $K$ be a contraction gauge where $\widehat{diam}_K(\hat{X}) = D<\infty$. Assume by contradiction that $D > \Delta$. There exist points $x, y\in X$ with $\hat{d}(x, y)\geq D - 2$. Let $m$ be the midpoint of the geodesic $[x, y]_K$. As depicted in Figure \ref{fig:diam}, there exists $g\in \mathrm{Isom}(X)$ such that $d(gx, m)\leq \rho+1$ (and hence $\hat{d}(gx, m)\leq \rho+1$). Consider the triangle with vertices $x, y$ and $gy$. Since $\hat{X}_K$ is a $\delta_0$ hyperbolic space, there exists $p\in [x, gy]_K\cup[gy, y]_K$ with $\hat{d}(m, p)\leq \delta_0$ (and hence $\hat{d}(p, gx)\leq \rho + \delta_0 +1$). If $p\in [x, gy]_K$, we have that 
\begin{align*}
    \hat{d}(x, gy) = \hat{d}(x, p) + \hat{d}(p, gy) \geq \hat{d}(x, m) - \delta_0 + \hat{d}(gy, gx) - \delta_0 - \rho - 1 \geq \frac{3}{2}D - 2\delta_0 - \rho - 4 > D, 
\end{align*}
a contradiction. If $p\in[gy, y]_K$, we get a contradiction analogously, which concludes the proof.
\end{proof}

Next we will prove a generalized version of the moreover part of Theorem \ref{thmintro:diameter_dich}. This is summarized in the proposition below.

\begin{prop}[Generalized version of moreover part of Theorem \ref{thmintro:diameter_dich}]\label{prop:diam_dichII}
    If $\mathrm{Isom}(X)$ acts coboundedly on $X$ and $K$ is a full contraction gauge, then $\hat{X}_K$ is bounded if and only if $X$ does not have any strongly contracting geodesic ray or equivalently, if $\mathrm{Isom(X)}\acts X$ does not have a strongly-contracting element.
\end{prop}

One direction follows directly from Lemma \ref{lemma:contracting_implies_quasi_geodesic}. To show the other we use the classification of actions on hyperbolic spaces.

As shown in Gromov \cite{Gromov-hyp} an action of a group $G$ on a hyperbolic space $Y$ falls into exactly one of the following three categories: 
\begin{itemize}
    \item (elliptic) The orbit of $G$ is bounded.
    \item (parabolic) The group $G$ has unbounded orbit but no loxodromic elements.
    \item (loxodromic) There exists a loxodromic element $g\in G$.
\end{itemize}

For example in \cite{CDMT:amenable}, it is shown that if a group $G$ acts coboundedly on a hyperbolic space $Y$, then the action cannot be parabolic, leading to the following well-known lemma.

\begin{lemma}\label{lemma:classification_hyperbolic_actions}
    If a group $G$ acts coboundedly on a hyperbolic space $Y$, then either $Y$ is bounded, or there exists an element $g\in G$ acting loxodromically on $Y$.
\end{lemma}

With this result, we are ready to prove Proposition \ref{prop:diam_dichII}.

\begin{proof}[Proof of Proposition \ref{prop:diam_dichII}]
    If $\mathrm{Isom}(X)$ has a strongly-contracting element, then $X$ contains a strongly contracting geodesic. Corollary \ref{corollary:unbounded_if_strongly_contracting} then states that the contraction space $\hat{X}_K$ is unbounded.

    On the other hand, if $\hat{X}_K$ is unbounded, then Lemma \ref{lemma:classification_hyperbolic_actions} implies that there exists a loxodromic element $g\in \mathrm{Isom}(X)$ (since $X$ is $1$-dense in $(\hat{X}_K, \hat{d})$, $\mathrm{Isom}(X)$ acts coboundedly on $(\hat{X}_K, \hat{d})$). Proposition \ref{lemma:loxodromic_implies_strongly_contracting} implies that $g$ is strongly-contracting, and hence that there exists a strongly-contracting bi-infinite geodesic in $X$.
\end{proof}

\begin{lemma}\label{lemma:undistorted-empty-subgroups}
    Let $H$ be an undistorted subgroup of a finitely generated group $G$. Assume that the Morse boundary of $H$ is empty (that is, no geodesic metric space $Y$ on which $H$ acts geometrically contains a Morse geodesic ray). Let $K$ be a full contraction gauge such that $K(r)\geq 10r +1$ for all $r\geq 0$. If $X$ is a geodesic metric space on which $G$ acts geometrically, then the orbit $Hx$ is bounded with respect to the metric $\hat{d}$ for all elements $x\in X$. 
\end{lemma}
\begin{proof}

    Let $x\in X$ and let $\tilde{Y}$ be the convex hull of the orbit $Hx\subset (\hat{X}_K, \hat{d})$. Since $\hat{X}_K$ is hyperbolic and $H\subset G$ is undistorted, $H$ acts coboundedly on $\tilde{Y}$ and $\tilde{Y}$ is hyperbolic. Assume by contradiction that $\widehat{\diam}(\tilde{Y}) = \infty$. By Lemma \ref{lemma:classification_hyperbolic_actions}, there exists an element $h\in H$ such that $h\acts \tilde{Y}$ is loxodromic. Consequently $h\acts \hat{X}_K$ is loxodromic. By Proposition \ref{lemma:loxodromic_implies_strongly_contracting}, $h\acts X$ is strongly contracting. Since $G$ acts geometrically on $X$, this implies that $h$ is Morse --- a contradiction to $H$ having empty Morse boundary.
\end{proof}

\begin{cor}
    Let $K$ be a full contraction gauge with $K(r)\geq 10r+1$ for all $r\geq 0$. If $Y\subset X$ is an undistorted subspace with empty Morse boundary and $\mathrm{Y}$ acts coboundedly on $Y$ with density $\rho$, then $\widehat{\mathrm{diam}}(Y)$ is bounded, where the bound only depends on the density $\rho$.
\end{cor}

\section{Relatively hyperbolic groups}\label{sec:relatively_hyperbolic_groups}

In this section, we show that for groups hyperbolic relative to groups with empty Morse boundary, the contraction space is naturally quasi-isometric to the coned off space. This suggests that the contraction space should be viewed as a generalization of the coned-off space for not necessarily hyperbolic groups.

This is the only section were the results do not hold for the alternative contraction space and the $(\mc G_K, K)$--contraction space as defined in Section \ref{sec:alternative-definitions} and one of the main reasons we chose the definition of anti-contracting as we have. 

\subsection{Definition and properties}

Relative hyperbolicity can be thought of as a generalisation of hyperbolicity. Roughly speaking, a group $G$ (space $X$) is hyperbolic relative to a collection of subgroups (subspaces) $\mc P$ if $G$ ($X$) is hyperbolic after ``collapsing'' all elements of $\mc P$. A precise definition can be found in \cite{Bow:rel-hyp}. There are many different but equivalent definitions of relatively hyperbolic groups (and spaces). In this paper, we do not work with any definition in particular but use the the results of \cite{S:distformrelhyp}, which we summarize below.

\begin{lemma}[Consequences of \cite{S:distformrelhyp}]\label{lemma:rel_hyp_projection}
Let $X$ be a geodesic metric space hyperbolic relative to a collection of subset $\mc P$ and let $\{\pi_P : X\to P\mid P \in \mc P\}$ be a collection of closest point projections. There exists a constant $C$ and an increasing
function $f : \R_{\geq 1}\to \R_{\geq 1}$ such that the following conditions hold. 
\begin{enumerate}
   \item  If $\gamma$ is a geodesic and $\diam (\gamma \cap \mc N_{C}(P))\leq C'$ for all peripherals $P\in \mc P$, then $\gamma$ is $f(C')$--contracting. \label{prop:small-intersection-implies-strong-contraction}
   \item For all $P\in \mc P$, $x\in X$ and $p\in P$, $d_X(p, \pi_P(x))\leq C+ (d_X(x, p) - d_X(x, P))$.\label{prop:projections-are-nice}
   \item Let $\gamma$ be a geodesic with endpoints $x, y\in X$ and let $P\in \mc P$ be a peripheral. If $d(\gamma, \pi_P(x))\geq C$, then $d(\pi_P(x), \pi_P(y))\leq C$.\label{prop:large-projection-implies-intersection}
   \item Let $\gamma$ be a geodesic and let $P\in \mc P$ be a peripheral. Let $x, y$ be the first and last points of $\gamma$ which are in the closed $C$-neighbourhood of $P$, then for any point $z\in \gamma_{[x, y]}$ we have that $d(z, P)\leq 2C$.\label{prop:convexity}
\end{enumerate}

\end{lemma}

The coned-off graph is defined as follows.

\begin{defn}[Coned-off graph]
    Let $G$ be a finitely generated group with generating set $S$ which is relatively hyperbolic to a collection of subgroups (subsets) $\mc P$. Let $X = \cay$. The coned-off graph $\tilde{X}$ is defined as follows. 

    \textbf{Vertices of $\tilde{X}$:} The vertices of $\tilde{X}$ are the elements of $G$.

    \textbf{Edges of $\tilde{X}$:} There is an edge between vertices $x, y \in G$ if $x$ and $y$ are connected by an edge in $X$ or if there exists a peripheral $P\in \mc P$ such that $x, y\in P$. 
\end{defn}

With this definition, there exists a natural inclusion from $X$ to the coned-off graph $\tilde{X}$, which we will denote by $\iota_{cone}$. We denote the (induced graph) metric of $\tilde{X}$ by $d_{cone}$.

\subsection{Quasi-isometry between the coned-off graph and the contraction space}

In this section, $G$ denotes a finitely generated group with finite generating set $S$ which is hyperbolic relative to a collection of subgroups $\mc P$ which all have empty Morse boundary. The set $\{\pi_P : X\to P | P\in \mc P\}$ denotes a system of closest point projections and $C, f$ are as in Lemma \ref{lemma:rel_hyp_projection}. They Cayley graph $\cay$ is denoted by $X$ and the coned-off graph by $\tilde{X}$.

We show that the coned-off graph $\tilde{X}$ is quasi-isometric to the contraction space $\hat{X}$ by showing that sets which are uniformly bounded in one graph are also uniformly bounded in the other.

\begin{lemma}\label{lem:paraboilc_subgroups_have_finite_diam}
    Let $P\in \mc P$ be a peripheral, then $\widehat{\diam}(P)<\infty$. 
\end{lemma}

\begin{proof}
    Since peripheral subgroups of relatively hyperbolic groups are undistorted (this is for example a consequence of Lemma \ref{lemma:rel_hyp_projection}\ref{prop:convexity}), the lemma is a direct consequence of Lemma \ref{lemma:undistorted-empty-subgroups}.
\end{proof}

Since there are only finitely many types of peripheral subgroups, we have that $\sup_{P\in \mc P}\{\widehat{\diam}{P}\}< \infty$ and we denote it by $D_{per}$.

The next lemma can be thought of as a converse of Lemma \ref{lem:paraboilc_subgroups_have_finite_diam} and shows that if points are far away in the coned-off graph $\tilde{X}$, they are far away in the contraction space $\hat{X}$.
 
\begin{lemma}\label{lemma:thin-transitions}
     There exists a constant $D_{rel}$ such that the following holds. Let $x, y\in X$ be two points with $d_{cone}(x, y) > D_{rel}$, then $\hat{d}(x, y) > 1$.
\end{lemma}

\begin{proof}
    
    We will first prove the following claim. 
    
    \begin{claim}\label{claim:2}
        For large enough $D$, if there exists a peripheral $P\in \mc P$ such that following hold, then $(x, y)$ is not anti-contracting.
        \begin{itemize}
            \item $\diam(\mc N_C(P)\cap [x, y] )\geq D$. Let $z\in [x, y]$ be the point closest to $x$ with $d(z, P)\leq C$.
            \item $d_{cone}(x, z)\geq D$.
        \end{itemize}
    \end{claim}
    \textit{Proof of Claim.}
        Let $\gamma$ be the subsegment of $[x, y]$ centred at $z$ with length $D/2$. Let $x'$ and $y'$ be the left and right endpoint of $\gamma$. Note that by Lemma \ref{lemma:rel_hyp_projection}\eqref{prop:convexity}, we have that $d(y', P)\leq 2C$ and hence by the triangle inequality, $d(\pi_P(z), \pi_P(y'))\geq D/2 - 3C$. Furthermore, $d(\gamma_{[x', z]}, P)\geq C$ by the definition of $z$. Hence $d(\pi_P(x'), \pi_P(z))\leq C$ by \eqref{prop:large-projection-implies-intersection}. 
        
        Let $\mc Q = (\gamma_1, \gamma_2, \gamma_3, \gamma_4)$ be a quadrangle containing $\gamma$. Let $a$ and $b$ be the left and right endpoints of $\gamma_1$. Observe that $d(\pi_P(a), \pi_P(x'))\leq C$ by \eqref{prop:large-projection-implies-intersection} and hence by the triangle inequality $d(\pi_P(a), z)\leq 3C$. Consequently, $d(\pi_P(a), \pi_P(b))\geq d(\pi_P(x'), \pi_P(y')) - 2C \geq D/2 - 5C$. Let $2\leq i\leq 4$ be the largest index such that $d(\pi_P(\gamma_i^+), \pi_P(\gamma_i^-)) > C$. If $D/2 - 2C\geq 3C +1$, there exists $4\geq i \geq 2$ which satisfies this. By Lemma \ref{lemma:rel_hyp_projection}\eqref{prop:large-projection-implies-intersection} we have that $d(\gamma_i, \pi_P(\gamma_i^+)) < C$ and by maximality of $i$, we have that $d(\gamma_i, z)\leq 2C + d(\pi_P(a), z) \leq 5C$.

        Hence, for $D\geq 50C +1$, we have that $\gamma$ is $5C$--thin and at least $50C+1$--long, implying that the pair $(x, y)$ is not anti-contracting. 
    \hfill$\blacksquare$

    From now on, fix $D$ large enough that satisfies the claim. Next we will show that for large enough $D_{rel}$, the statement holds. Conditions for $D_{rel}$ are outlined in the proof. Let $x, y\in X$ be points with $\tilde{d}(x, y)\geq D_{rel}$. 

    If $[x, y]$ satisfies the assumptions of Claim \ref{claim:2}, then $(x, y)$ is not anti-contracting and we are done. Assume instead that $[x, y]$ does not satisfy the assumptions of Claim \ref{claim:2}. In other words, for all peripherals $P\in \mc P$, at least one of the following holds.
    \begin{itemize}
        \item $d_{cone}(x, P)\leq D + C$,
        \item $\diam(\mc N_C(P), [x, y])\leq D$.
    \end{itemize}
    Let $z\in [x, y]$ be the last point on $[x, y]$ with $d_{cone}(x, z) = D + 2C+1$. The conditions above imply that for every peripheral $\diam(\mc N_C(P), [x, y]_{[z, y]})\leq D$. By Lemma \ref{lemma:rel_hyp_projection}\eqref{prop:small-intersection-implies-strong-contraction}, we have that $[x, y]_{[z, y]}$ is $f(D)$--contracting. We have that $d(z, y)\geq d_{cone}(z, y)\geq D_{rel} - D - 2C -1$. So, for large enough $D_{rel}$, the geodesic $[x, y]_{[z, y]}$ is $r = \Phi(f(D))$--quadrangle-contracting and at least $10r+1$--long, implying that $(x, y)$ is not anti-contracting and hence $\hat{d}(x, y)>1$.
\end{proof}

We are now ready to prove that the coned-off graph $\tilde{X}$ and the contraction space $\hat{X}$ are quasi-isometric. 

\begin{prop}\label{prop:quasi-isom-to-cone}
     Let $G$ be a finitely generated group which is hyperbolic relative to a collection of subgroups $\mc P$ which all have empty Morse boundary. Then the contraction space $\hat{X}$ is naturally quasi-isometric to the coned off graph $\tilde{X}$.
\end{prop}

\begin{proof}
    Let $x, y\in G$ and let $l = d_{cone}(x, y)$. Let $\gamma: I\to \tilde{X}$ be a geodesic from $x$ to $y$. For $0\leq i \leq l$, denote by $x_i$ the point on $\gamma$ with $d_{cone}(x, x_i) = i$. With this, $x_0 = x$ and $x_l = y$. For each $i$, we have one of the following. 
    \begin{itemize}
        \item  $(x_i, x_{i+1})\in E(X)$. In this case, $\hat{d}(x_i, x_{i+1})\leq d(x_i, x_{i+1})\leq 1$.
        \item There exists a peripheral $P\in \mc P$ such that $x_i, x_{i+1}\in P$. In this case, $\hat{d}(x_i, x_{i+1})\leq D_{per}$ by Lemma \ref{lem:paraboilc_subgroups_have_finite_diam} and the remark below.
    \end{itemize}
    
    Thus, by the triangle inequality, $\hat{d}(x, y)\leq lD_{per}$.

    Conversely, let $x, y\in X$ be points with $\hat{d}(x, y) \leq l$ and let $\gamma: I\to \hat{X}$ be a geodesic from $x$ to $y$. Define $x_0 = x$. For $1\leq i \leq 2l$, inductively define $x_i$ as the point on $\gamma\cap X$ which maximizes $\hat{d}(x, x_{i})$ but satisfies $\hat{d}(x_{i-1}, x_i)\leq 1$. Note that if $\hat{d}(x_i, x_{i-1})< 1$, then either $x_i = y$ or the points right after $x_i$ were in $\hat{X}-X$. Since $\hat{X}-X$ consists of a disjoint union of interiors of edges of length 1, the latter implies that $\hat{d}(x_i, x_{i+1}) = 1$. Thus, we have that $x_{2l} = y$. Further, for each $1\leq i \leq 2l$ we have by Lemma \ref{lemma:thin-transitions}, that $d_{cone}(x_{i-1}, x_i)\leq D_{rel}$. By the triangle inequality, we have that $d_{cone}(x, y)\leq 2lD_{rel}$.

    The above observations show that the map $\hat{\iota}_{cone} : (X, \hat{d})\to (\tilde{X}, d_{cone})$ is a quasi-isometry. Since $(X, \hat{d})$ is 1-dense in $\hat{X}$, the proposition follows. 
\end{proof}

\bibliography{mybib}
\bibliographystyle{alpha}

\end{document}